\documentclass[11pt]{amsproc}
\usepackage[british]{babel}

\usepackage{a4wide}
\usepackage{setspace}
\usepackage{amssymb}
\usepackage{amsmath}
\usepackage{amsthm}
\usepackage{enumitem}
\usepackage{mathrsfs}
\usepackage[colorlinks,citecolor=blue,urlcolor=black,linkcolor=red,linktocpage]{hyperref}

\newtheorem{thm}{Theorem}[section]

\newtheorem{cor}[thm]{Corollary}
\newtheorem{lem}[thm]{Lemma}

\newtheorem{prop}[thm]{Proposition}

\theoremstyle{definition}

\newtheorem{definition}[thm]{Definition}

\newtheorem{remark}[thm]{Remark}

\renewcommand{\epsilon}{\varepsilon}
\renewcommand{\phi}{\varphi}
\newcommand{\defeq}{\mathrel{\mathop:}=}

\DeclareMathOperator{\gr}{gr}

\DeclareMathOperator{\id}{id}
\DeclareMathOperator{\dom}{dom}
\DeclareMathOperator{\spt}{spt}

\DeclareMathOperator{\Aut}{Aut}
\DeclareMathOperator{\Iso}{Iso}

\makeatletter
\def\moverlay{\mathpalette\mov@rlay}
\def\mov@rlay#1#2{\leavevmode\vtop{%
		\baselineskip\z@skip \lineskiplimit-\maxdimen
		\ialign{\hfil$\m@th#1##$\hfil\cr#2\crcr}}}
\newcommand{\charfusion}[3][\mathord]{
	#1{\ifx#1\mathop\vphantom{#2}\fi
		\mathpalette\mov@rlay{#2\cr#3}
	}
	\ifx#1\mathop\expandafter\displaylimits\fi}
\makeatother

\newcommand{\cupdot}{\charfusion[\mathbin]{\cup}{\cdot}}
\newcommand{\bigcupdot}{\charfusion[\mathop]{\bigcup}{\cdot}}

\begin{document}

\onehalfspace

\setlist{noitemsep}

\author{Friedrich Martin Schneider}
\address{F.M.S., Institute of Algebra, TU Dresden, 01062 Dresden, Germany }
\curraddr{Department of Mathematics, University of Auckland, Private Bag 92019, NZ}
\email{martin.schneider@tu-dresden.de}

\author{Andreas Thom}
\address{A.T., Institute of Geometry, TU Dresden, 01062 Dresden, Germany }
\email{andreas.thom@tu-dresden.de}

\title{On F\o lner sets in topological groups}
\date{\today}

\begin{abstract} 
  We extend F\o lner's amenability criterion to the realm of general topological groups. Building on this, we show that a topological group $G$ is amenable if and only if its left translation action can be approximated in a uniform manner by amenable actions on the set~$G$. As applications we obtain a topological version of Whyte's geometric solution to the von Neumann problem and give an affirmative answer to a question posed by Rosendal.
\end{abstract}


\maketitle


\tableofcontents

\section{Introduction}

This paper is a continuation of the study and application of F\o lner-type characterizations of amenable topological groups that we initiated in \cite{SchneiderThom}. The study of amenability of discrete groups benefits from a wealth of possible viewpoints -- ranging from analytic to combinatorial --  that allow for numerous approaches to problems and also many surprising applications. The study of amenable groups was started by von Neumann in his analysis of the Banach-Tarski paradox. Since then, the distinction between amenability, hyperfiniteness, and almost invariance on one side and non-amenability, paradoxicality, and freeness on the other side has been a recurring theme in modern mathematics -- appearing in group theory, functional analysis, ergodic theory, and operator algebras.
In this paper, we will extend the study of this dichotomy in the context of general topological groups, derive various non-trivial consequences and develop a new more combinatorial point of view towards amenable topological groups. The key concept in our study is a suitable notion of F\o lner set. F\o lner's insight \cite{folner} was that the existence of almost invariant finite subsets of a discrete group, a condition that was obviously sufficient for the existence of an invariant mean, is also necessarily satisfied. This clarified the situation and opened the way to various fundamental applications -- let us just mention the development of Ornstein-Weiss entropy theory for group actions of amenable groups. We follow a similar route and identify a topological matching condition that characterizes amenable topological groups.

In this paper we will lay the foundation of a F\o lner-type combinatorial characterization of amenability and clarify its relationship with paradoxicality and freeness -- applications towards an entropy theory of group actions will be subject of forthcoming work. One of our main results is a generalization of a result of Whyte that characterizes finitely generated non-amenable discrete group by the existence of a partition of the Cayley graph into Lipschitz embedded trees, or equivalently, the existence of a semi-regular free subgroup inside the associated wobbling group. This way we can characterize amenability of a topological group $G$ by existence of arbitrarily small perturbations of the left-translation  action by amenable actions on  $G$ as a {\it set}. More precisely, but still somewhat informally, if $G$ is an amenable Polish group, then there exists a sequence of maps $\alpha_n \colon G \to {\rm Sym}(G)$, so that the image of each $\alpha_n$ preserves a mean on the set $G$, and $\alpha_n(g)(h) \to gh$ as $n \to \infty$ for all $g,h \in G$ uniformly. If $G$ is compact, we can approximate the left-translation action of $G$ on itself by finite subgroups of ${\rm Sym}(G)$, i.e., not only does the image of $\alpha_n$ preserve a mean on the set $G$, but it generates a finite subgroup of ${\rm Sym}(G)$.

The focus of the final part of the paper is towards a study of the coarse geometry of amenable topological groups -- as recently initiated by Rosendal \cite{rosendal2}. Our theory provides a natural setup in which some classical constructions, that rely on the existence of F\o lner sets in
the realm of locally compact groups, can be carried out without further problems. Using F\o lner sets, we can perform a certain ultra-product argument to see, that any amenable topological group that embeds coarsely into some Banach space $E$, embeds also coarsely and equivariantly into a Banach space $V$ that is finitely representable in $L^p(E)$.

\hspace{0.1cm}

This paper is organized as follows. In Section~\ref{section:uniform.spaces} we recollect some facts about uniform spaces and UEB topologies. In Section~\ref{section:day} we prove an amenability criterion for topological groups by means of almost invariant vectors (Theorem~\ref{theorem:topological.day}), and in Section~\ref{section:folner} we give a corresponding characterization in terms of topological F\o lner sets (Theorem~\ref{theorem:topological.folner}). Utilizing these results, in Section~\ref{section:whyte} we prove the above-mentioned approximation result concerning perturbed translations on amenable topological groups (Theorem~\ref{theorem:approximating.actions}) and deduce a topological version of Whyte's geometric solution to the von Neumann problem (Corollary~\ref{corollary:von.neumann}). Our final Section~\ref{section:equivariant.geometry} is devoted to an application of our results to coarse geometry of topological groups: in fact, we give an affirmative answer to a question posed by Rosendal (Theorem~\ref{theorem:rosendal}).

\section{Uniform spaces and the UEB topology}\label{section:uniform.spaces}

In this section we briefly review some preliminaries concerning uniform spaces and the UEB topology. For further reading on this subject we refer to~\cite{PachlBook}.

For the sake of convenience, we recall some basic definitions and facts concerning uniform spaces. A \emph{uniformity} on a set $X$ is a filter $\mathscr{E}$ on the set $X \times X$ such that \begin{enumerate}
	\item[$\bullet$] $\{ (x,x) \mid x \in X \} \subseteq E$ for every $E \in \mathscr{E}$,
	\item[$\bullet$] $E^{-1} \in \mathscr{E}$ for every $E \in \mathscr{E}$,
	\item[$\bullet$] for every $E_{0} \in \mathscr{E}$ there exists $E_{1} \in \mathscr{E}$ such that $E_{1} \circ E_{1} \subseteq E_{0}$.
\end{enumerate} A \emph{uniform space} is a set $X$ endowed with a uniformity on $X$, whose elements are called the \emph{entourages} of the uniform space $X$. For a uniform space $X$, the \emph{induced topology} on $X$ is defined as follows: a subset $S \subseteq X$ is \emph{open} in $X$ if, for every $x \in S$, there is an entourage $E$ of $X$ such that $\{ y \in X \mid (x,y) \in E \} \subseteq S$. Let $X$ and $Y$ be uniform spaces. A map $f \colon X \to Y$ is called \emph{uniformly continuous} if for every entourage $E$ of~$Y$ there exists some entourage $F$ of~$X$ such that $(f \times f)(F) \subseteq E$. A bijection $f \colon X \to Y$ is called an \emph{isomorphism} if both $f$ and $f^{-1}$ are uniformly continuous maps. Moreover, a set $H \subseteq Y^{X}$ is called \emph{uniformly equicontinuous} if for every entourage $E$ of $Y$ there exists some entourage $F$ of~$X$ such that $(f \times f)(F) \subseteq E$ whenever $f \in H$. If $I$ is a set, then the \emph{uniformity of uniform convergence} on $X^{I}$ is defined to be the least uniformity on $X^{I}$ containing all the sets of the form \begin{displaymath}
	\left. \left\{ (f,g) \in X^{I} \times X^{I} \, \right\vert \forall i \in I \colon \, (f(i),g(i)) \in E \right\} \qquad (E \text{ entourage of } X) ,
\end{displaymath} and the corresponding \emph{topology of uniform convergence} on $X^{I}$ is defined to be the topology induced by that uniformity. Evidently, there is a natural isomorphism between the resulting uniform spaces $X^{I \times J}$ and $(X^{I})^{J}$ for any two sets $I$ and $J$.

Given a pseudo-metric $d$ on a set $X$, we define \begin{displaymath}
	B_{d}(x,\epsilon) \defeq \{ y \in X \mid d(x,y) < \epsilon \} , \qquad B_{d}[x,\epsilon] \defeq \{ y \in X \mid d(x,y) \leq \epsilon \} 
\end{displaymath} for $x \in X$ and $\epsilon > 0$. A pseudo-metric $d$ on a topological space $X$ is said to be \emph{continuous} if $B_{d}(x,\epsilon)$ is open in $X$ for all $x \in X$ and $\epsilon > 0$, that is, the topology generated by $d$ is contained in the topology of~$X$. Furthermore, a pseudo-metric $d$ on a uniform space $X$ is called \emph{uniformly continuous} if the uniformity generated by $d$, i.e., \begin{displaymath}
	\mathscr{E}_{d} \defeq \{ E \subseteq X \times X \mid \exists \epsilon > 0 \, \forall x,y \in X \colon \, d(x,y) < \epsilon \Longrightarrow (x,y) \in E \} ,
\end{displaymath} is contained in the uniformity of~$X$. It was shown by Weil~\cite{weil} that every uniformity on a set $X$ is the union of the directed family of uniformities generated by the corresponding uniformly continuous pseudo-metrics on $X$.

We continue with some remarks concerning the UEB topology (cf.~\cite{PachlBook,NeufangPachlPekka,PachlSteprans}). As usual, given a set $X$, we will denote by $\ell^{\infty}(X)$ the Banach space of all bounded real-valued functions on $X$ equipped with the supremum norm \begin{displaymath}
	\Vert f \Vert_{\infty} \defeq \sup \{ \vert f(x) \vert \mid x \in X \} \qquad (f \in \ell^{\infty}(X)) .
\end{displaymath} Let $X$ be a uniform space. The set $\mathrm{UC}_{b}(X)$ of all bounded uniformly continuous real-valued functions on $X$ is a closed linear subspace of $\ell^{\infty}(X)$ and thus constitutes a Banach space itself when endowed with the supremum norm. A subset $H \subseteq \mathrm{UC}_{b}(X)$ is called \emph{UEB} (short for \emph{uniformly equicontinuous bounded}) if $H$ is uniformly equicontinuous and $\Vert \cdot \Vert_{\infty}$-bounded. It is not difficult to see that a subset $H \subseteq \mathrm{UC}_{b}(X)$ is UEB if and only if $H$ is $\Vert \cdot \Vert_{\infty}$-bounded and there exists $\ell \geq 0$ along with a uniformly continuous pseudo-metric $d$ on $X$ such that every member of $H$ is $\ell$-Lipschitz continuous for $d$. The \emph{UEB topology} on the continuous dual~$\mathrm{UC}_{b}(X)'$ is defined as the topology of uniform convergence on UEB subsets of $\mathrm{UC}_{b}(X)$. This is a locally convex linear topology on the vector space $\mathrm{UC}_{b}(X)'$, which by our remark above coincides with the topology defined by the semi-norms \begin{displaymath}
	p_{d}(\mu) \defeq \sup \{ \vert \mu(f) \vert \mid f \colon (X,d) \to [-1,1] \text{ $1$-Lipschitz} \} \qquad (\mu \in \mathrm{UC}_{b}(X)') 
\end{displaymath} where $d$ runs trough all uniformly continuous pseudo-metrics on~$X$. Since the set of uniformly continuous pseudo-metrics on $X$ is upwards directed with respect to point-wise ordering and we have $p_{d} \leq p_{d'}$ for any two uniformly continuous pseudo-metrics $d$ and $d'$ on $X$ with $d \leq d'$, it follows that \begin{displaymath}
	\{ B_{p_{d}}(\mu,\epsilon) \mid d \text{ uniformly continuous pseudo-metric on } X, \, \epsilon > 0 \}
\end{displaymath} constitutes a neighborhood basis at any $\mu \in \mathrm{UC}_{b}(X)'$ with regard to the UEB topology.

Before we continue, let us clarify some notation. Let $X$ be a set. We denote by $\mathscr{F}(X)$ the set of all finite subsets of $X$. For a function $f \colon X \to \mathbb{R}$, we put $\spt (f) \defeq \{ x \in X \mid f(x) \ne 0\}$. Consider the real vector space $\mathbb{R}X \defeq \{ f \in \mathbb{R}^{X} \mid \spt (f) \text{ finite} \}$. Moreover, let \begin{displaymath}
	\Vert f \Vert_{1} \defeq \sum_{x \in X} \vert f(x) \vert \qquad (f \in \mathbb{R}X) .
\end{displaymath} For $x \in X$, define $\delta_{x} \in \mathbb{R}X$ by setting $\delta_{x}(x) \defeq 1$ and $\delta_{x}(y) \defeq 0$ for every $y \in X$ with $y \ne x$. Let $\delta_{F} \defeq \frac{1}{\vert F \vert} \sum_{x \in F} \delta_{x}$ for any finite non-empty $F \subseteq X$. For $a \in \mathbb{R}X$ and $f \in \mathbb{R}^{X}$, let \begin{displaymath}
	a(f) \defeq \sum_{x \in X} a(x)f(x) .
\end{displaymath}

Now let $X$ be a uniform space again. Then the line above defines a linear map from $\mathbb{R}X$ into the continuous dual $\mathrm{UC}_{b}(X)'$, which is injective if and only if $X$ is Hausdorff. In any case, it allows us to pull back the UEB topology and the semi-norms defined above onto $\mathbb{R}X$. In terms of notation, we will not distinguish between the semi-norms on $\mathrm{UC}_{b}(X)'$ and the corresponding ones on $\mathbb{R}X$. As the following proposition reveals, the continuous dual of the locally convex topological vector space $\mathbb{R}X$ may be identified with $\mathrm{UC}_{b}(X)$.

\begin{prop}[cf.~Lemma~6.5 and Theorem~6.6 in~\cite{PachlBook}]\label{proposition:duality} If $X$ is a uniform space, then the map $\Phi \colon \mathrm{UC}_{b}(X) \to \mathbb{R}X'$ given by \begin{displaymath}
	\Phi (f)(a) \defeq a(f) \qquad (f \in \mathrm{UC}_{b}(X), \, a \in \mathbb{R}X)
\end{displaymath} is an isomorphism of real vector spaces. \end{prop}

\begin{proof} We first check that $\Phi$ is well defined. To this end, let $f \in \mathrm{UC}_{b}(X)$. A standard computation shows that $\Phi (f) \colon \mathbb{R}X \to \mathbb{R}$ is linear. We need to show that $\Phi (f)$ is continuous. Without loss of generality, we may assume that $\Vert f \Vert_{\infty} \leq 1$. As $f$ is uniformly continuous, we obtain a uniformly continuous pseudo-metric $d$ on $X$ by setting \begin{displaymath}
	d(x,y) \defeq \vert f(x) - f(y) \vert \qquad (x,y \in X) .
\end{displaymath} Let $\epsilon > 0$. Since $f \colon (X,d) \to [-1,1]$ is $1$-Lipschitz continuous, we have $\vert \Phi (f)(a) \vert = \vert a(f) \vert \leq \epsilon$ for every $a \in B_{p_{d}}[0,\epsilon]$. Hence, $\Phi (f)(B_{p_{d}}[0,\epsilon]) \subseteq [-\epsilon,\epsilon]$. So, $\Phi$ is indeed well defined.

Moreover, $\Phi$ is obviously a linear map. Clearly, $\Phi$ is also injective: if $f \in \mathrm{UC}_{b}(X)$ and $f \ne 0$, then there exists some $x \in X$ with $f(x) \ne 0$, and we conclude that $\Phi (f)(\delta_{x}) = f(x) \ne 0$ and hence $\Phi (f) \ne 0$. It remains to show that $\Phi$ is surjective. For this purpose, consider a continuous linear functional $F \colon \mathbb{R}X \to \mathbb{R}$. Define $f \colon X \to \mathbb{R}, \, x \mapsto F(\delta_{x})$. Now, $f$ is bounded for the following reason: as $F$ is continuous, there is a uniformly continuous pseudo-metric $d$ on $X$ such that $F(B_{p_{d}}[0,1])$ is bounded, and since $\{ \delta_{x} \mid x \in X \} \subseteq B_{p_{d}}[0,1]$, it follows that \begin{displaymath}
	\sup \{ \vert f(x) \vert \mid x \in X \} = \sup \{ \vert F(\delta_{x}) \vert \mid x \in X \} < \infty .
\end{displaymath} To see that $f$ is uniformly continuous, let $\epsilon > 0$. By continuity of $F$, there exists a uniformly continuous pseudo-metric $d$ on $X$ along with some $\delta > 0$ such that $F(B_{p_{d}}[0,\delta]) \subseteq [-\epsilon,\epsilon]$. If $x,y \in X$ with $d(x,y) \leq \delta$, then $p_{d}(\delta_{x} - \delta_{y}) \leq \delta$ and thus \begin{displaymath}
	\vert f(x) - f(y) \vert = \vert F(\delta_{x}) - F(\delta_{y}) \vert = \vert F(\delta_{x} - \delta_{y}) \vert \leq \epsilon .
\end{displaymath} This shows that $f \colon X \to \mathbb{R}$ is uniformly continuous. Finally, we observe that \begin{displaymath}
	\Phi (f)(a) = \sum_{x \in X} a(x)f(x) = \sum_{x \in X} a(x)F(\delta_{x}) = F\left( \sum_{x \in X} a(x)\delta_{x} \right) = F(a)
\end{displaymath} for all $a \in \mathbb{R}X$. So, $\Phi$ is surjective. This completes the proof. \end{proof}

The following lemma is frequently useful to pass from positive and normalized elements in $\mathbb{R}X$ to finite subsets of a space $X$.

\begin{lem}\label{lemma:density} If $X$ is a perfect Hausdorff uniform space, then $\{ \delta_{F} \mid F \in \mathscr{F}(X), \, F \ne \emptyset \}$ is dense in $\{ a \in \mathbb{R}X \mid \Vert a \Vert_{1} = 1, \, a \geq 0 \}$ with respect to the UEB topology. \end{lem}

\begin{proof} Let $a \in \mathbb{R}X$ with $\Vert a \Vert_{1} = 1$ and $a \geq 0$. Put $S \defeq \spt (a)$. Furthermore, let $\epsilon > 0$ and let $d$ be a uniformly continuous pseudo-metric on $X$. Since $\mathbb{Q}$ is dense in $\mathbb{R}$, there exists $b \colon X \to \mathbb{Q}$ such that $\spt (b) = S$, $b \geq 0$, $\Vert b \Vert_{1} = 1$, and $\Vert a-b \Vert_{1} \leq \frac{\epsilon}{2}$. Now, let $c \colon X \to \mathbb{N}$ and $n \geq 1$ such that $b(x) = \frac{c(x)}{n}$ for all $x \in X$. Since $X$ is a perfect Hausdorff space, every open non-empty subset of $X$ is infinite. For any $x \in X$, the continuity of $d$ implies that $B_{d}(x,\frac{\epsilon}{2})$ is a neighborhood of $x$ in $X$ and thus infinite. Hence, there is a map $\phi \colon S \to \mathscr{F}(X)\setminus \{ \emptyset \}$ with \begin{enumerate}
	\item[$(1)$]	$\phi(x) \subseteq B_{d}(x,\frac{\epsilon}{2})$ for every $x \in S$,
	\item[$(2)$]	$|\phi(x)| = c(x)$ for every $x \in S$,
	\item[$(3)$]	$\phi(x) \cap \phi(y) = \emptyset$ for any two distinct $x,y \in S$.
\end{enumerate} By~(2) and~(3), we have $|F| = n$ for $F \defeq \bigcup \{ \phi(x) \mid x \in S \}$. Moreover, if $f \colon (X,d) \to [-1,1]$ is $1$-Lipschitz continuous, then~(1) implies that \begin{align*}
	\vert (b - \delta_{F})(f) \vert \, &= \, \frac{1}{n} \left| \sum_{x \in S} c(x)f(x) - \sum_{x \in S}\sum_{y \in \phi(x)} f(y) \right| \, \leq \, \frac{1}{n} \sum_{x \in S} \left| c(x) f(x) - \sum_{y \in \phi (x)} f(y) \right| \\
	&\leq \, \frac{1}{n} \sum_{x \in S} \sum_{y \in \phi (x)} |f(x) - f(y)| \, \leq \, \frac{1}{n} \sum_{x \in S} \sum_{y \in \phi (x)} d(x,y) \, \leq \, \frac{\epsilon}{2n} \sum_{x \in S} c(x) \, = \, \frac{\epsilon}{2} ,
\end{align*} and therefore \begin{align*}
	\vert (a - \delta_{F})(f) \vert &\leq \vert (a-b)(f) \vert + \vert (b-\delta_{F})(f) \vert \leq \Vert a-b \Vert_{1} + \vert (b-\delta_{F})(f) \vert \leq \epsilon .
\end{align*} Hence, $\delta_{F} \in B_{p_{d}}[a,\epsilon ]$. This finishes the proof. \end{proof}

\section{Amenability and almost invariant vectors}\label{section:day}

Viewing a topological group $G$ as a uniform space and endowing the convolution algebra $\mathbb{R}G$ with the corresponding UEB topology, one might wonder whether amenability of $G$ can be characterized in terms of the existence of almost invariant vectors in $\mathbb{R}G$. In this section we establish such a criterion for amenability of topological groups (Theorem~\ref{theorem:topological.day}). 

For convenience, let us once again fix some notation. Let $G$ be any group. For $g \in G$, let $\lambda_{g} \colon G \to G, \, x \mapsto gx$ and $\rho_{g} \colon G \to G, \, x \mapsto xg^{-1}$. As usual, the convolution on $\mathbb{R}G$ is given by \begin{displaymath}
	ab \defeq \sum_{g,h \in G} a(g)b(h)\delta_{gh} \qquad (a,b \in \mathbb{R}G) .
\end{displaymath} We note that $\spt (ab) \subseteq (\spt a)(\spt b)$ for any two $a,b \in \mathbb{R}G$. Of course, $G \to \mathbb{R}G, \, g \mapsto \delta_{g}$ constitutes an embedding of $G$ into the multiplicative semigroup of the unital ring $\mathbb{R}G$. In particular, we have left and right actions of $G$ by linear transformations on $\mathbb{R}G$ given by \begin{displaymath}
	ga \defeq \delta_{g}a , \qquad ag \defeq a\delta_{g} \qquad (g \in G, \, a \in \mathbb{R}G) .
\end{displaymath} For $a \in \mathbb{R}G$ and $f \in \mathbb{R}^{G}$, let us also define $R_{a}f \in \mathbb{R}^{G}$ by \begin{displaymath}
	(R_{a}f)(x) \defeq a(f \circ \lambda_{x}) \qquad (x \in G) ,
\end{displaymath} that is, $R_{a}f = \sum_{g \in G} a(g^{-1}) (f \circ \rho_{g})$. Clearly, $R_{a}f \in \mathbb{R}G$ if $a,f \in \mathbb{R}G$. Furthermore, simple computations reveal that $R_{ab}f = R_{a}(R_{b}f)$ and $(ab)(f) = a(R_{b}f)$ for all $f \in \mathbb{R}^{G}$ and $a,b \in \mathbb{R}G$.

Throughout the present paper, unless explicitly stated otherwise, whenever a topological group $G$ is considered as a uniform space, we will be referring to its \emph{right uniformity}, i.e., \begin{displaymath}
	\mathscr{E}_{r}(G) \defeq \{ E \subseteq G \times G \mid \exists U \in \mathscr{U}_{e}(G) \, \forall x,y \in G \colon \, xy^{-1} \in U \Longrightarrow (x,y) \in E \} ,
\end{displaymath} where $\mathscr{U}_{e}(G)$ denotes the neighborhood filter of the neutral element in $G$. However, to avoid any possible confusion, we will denote the set of all bounded right-uniformly continuous real-valued functions on $G$ by $\mathrm{RUC}_{b}(G)$. Of course, the topology induced by the right uniformity of a topological group coincides with its original topology.

As usual, we will call a pseudo-metric $d$ on a group $G$ \emph{right-invariant} (\emph{left-invariant}, resp.) if $d(xg,yg) = d(x,y)$ ($d(gx,gy) = d(x,y)$, resp.) for all $g,x,y \in G$. Note that a right-invariant pseudo-metric on a topological group is continuous if and only if it is uniformly continuous with respect to the right uniformity. Moreover, if $G$ is a topological group and $H$ is a UEB subset of $\mathrm{RUC}_{b}(G)$, then \begin{displaymath}
	d(x,y) \defeq \sup_{g \in G} \sup_{f \in H} \vert f(xg) - f(yg) \vert \qquad (x,y \in G)
\end{displaymath} is a right-invariant continuous pseudo-metric on~$G$, and clearly $f \colon (G,d) \to \mathbb{R}$ is $1$-Lipschitz continuous for every $f \in H$. Since the set of right-invariant continuous pseudo-metrics on~$G$ is upwards directed with respect to point-wise ordering, it follows that \begin{displaymath}
	\{ B_{p_{d}}(a,\epsilon) \mid d \text{ right-invariant continuous pseudo-metric on } G, \, \epsilon > 0 \}
\end{displaymath} constitutes a neighborhood basis at any $a \in \mathbb{R}G$ with regard to the UEB topology. According to a recent result by Pachl and Stepr{\= a}ns~\cite{PachlSteprans}, the convolution on $\mathbb{R}G$ is jointly continuous with respect to the UEB topology if and only if $G$ is a \emph{SIN group}, i.e., its neutral element admits a neighborhood basis consisting of sets invariant under conjugation. However, convolution is always separately continuous: for every $a \in \mathbb{R}G$, the linear maps $\mathbb{R}G \to \mathbb{R}G, \, b \mapsto ab$ and $\mathbb{R}G \to \mathbb{R}G, \, b \mapsto ba$ are continuous with regard to the UEB topology. Moreover, if $H$ is any $\Vert \cdot \Vert_{1}$-bounded subset of the positive cone of $\mathbb{R}G$, then $\{ b \mapsto ba \mid a \in A \}$ is even uniformly equicontinuous. This is due to the following elementary fact.

\begin{lem}\label{lemma:contractive.functions} Let $d$ be a right-invariant pseudo-metric on a group $G$ and let $f \colon (G,d) \to \mathbb{R}$ be $1$-Lipschitz continuous. Then $R_{a}f \colon (G,d) \to \mathbb{R}$ is $1$-Lipschitz continuous for every $a \in \mathbb{R}G$ with $\Vert a \Vert_{1} = 1$ and $a \geq 0$. \end{lem}

\begin{proof} Let $a \in \mathbb{R}G$ with $\Vert a \Vert_{1} = 1$ and $a \geq 0$. For all $x,y \in G$, \begin{displaymath}
	\vert R_{a}f(x) - R_{a}f(y) \vert \, \leq \, \sum_{g \in G} a(g) \vert f(xg) - f(yg) \vert \, \leq \, \sum_{g \in G} a(g) d(xg,yg) \, = \, d(x,y) .\qedhere
\end{displaymath} \end{proof}

Now let us turn our attention towards amenability. Recall that a topological group $G$ is \emph{amenable} if $\mathrm{RUC}_{b}(G)$ admits a left-invariant mean, i.e., there exists a positive linear map $\mu \colon \mathrm{RUC}_{b}(G) \to \mathbb{R}$ such that $\mu (\mathbf{1}) = 1$ and $\mu (f \circ \lambda_{g}) = \mu (f)$ for all $f \in \mathrm{RUC}_{b}(G)$ and~$g \in G$. According to a well-known result of Rickert~\cite[Theorem~4.2]{rickert}, a topological group $G$ is amenable if and only if every continuous action of $G$ by affine homeomorphisms on a non-void compact convex subset of a locally convex topological vector space admits a fixed point. Strengthening this condition, a topological group $G$ is said to be \emph{extremely amenable} if every continuous action of $G$ on a non-empty compact Hausdorff space has a fixed point.

Our first main result relates amenability of a topological group $G$ to the existence of almost invariant vectors in the locally convex space $\mathbb{R}G$ carrying the UEB topology. It generalizes a result by Day~\cite{day57} for discrete groups (see also~\cite{namioka}).

\begin{thm}\label{theorem:topological.day} A topological group $G$ is amenable if and only if, for every $\epsilon > 0$, every finite subset $E \subseteq G$, and every right-invariant continuous pseudo-metric $d$ on $G$, there exists $a \in \mathbb{R}G$ with $\Vert a \Vert_{1} = 1$ and $a \geq 0$ such that \begin{displaymath}
	\forall g \in E \colon \quad p_{d}(a - ga) \leq \epsilon .
\end{displaymath} \end{thm}

\begin{proof} ($\Longrightarrow$) Our proof proceeds by contradiction. To this end, let us consider the set $T \defeq \{ a \in \mathbb{R}G \mid \Vert a \Vert_{1} = 1, \, a\geq 0 \}$. Suppose that there exist $\epsilon > 0$, a finite subset $E \subseteq G$ and a right-invariant continuous pseudo-metric $d$ on $G$ such that $\sup_{g \in E} p_{d}(a-ga) > \epsilon$ for all $a \in T$. Then $0$ is not contained in the closure of the convex subset $\{ (a - ga)_{g \in E} \mid a \in T \}$ in the locally convex space $(\mathbb{R}G)^{E}$. Applying the separation theorem for locally convex spaces, we conclude that there exists a continuous linear functional $F \in ((\mathbb{R}G)^E)'$ such that $F((a-ga)_{g \in E}) \geq 1$ for all $a \in T$. Since $((\mathbb{R}G)^E)' \cong ((\mathbb{R}G)')^{E}$, Proposition~\ref{proposition:duality} asserts the existence of $f_{g} \in \mathrm{RUC}_{b}(G)$ $(g \in E)$ such that $\sum_{g \in E} (a-ga)(f_{g}) \geq 1$ for all $a \in T$. That is, \begin{displaymath}
	a\left(\sum\nolimits_{g \in E} f_{g} - f_{g} \circ \lambda_{g}\right) = \sum\nolimits_{g \in E} a(f_{g}) - a(f_{g} \circ \lambda_{g}) = \sum\nolimits_{g \in E} (a-ga)(f_{g}) \geq 1
\end{displaymath} for all $a \in T$. Since $\{ \delta_{x} \mid x \in G \} \subseteq T$, it follows that $\sum_{g \in E} f_{g} - f_{g} \circ \lambda_{g} \geq 1$. Now, if $\mu$ was a left-invariant mean on $\mathrm{RUC}_{b}(G)$, then $\mu ( \sum_{g \in E} f_{g} - f_{g} \circ \lambda_{g}) \geq 1$, but also \begin{displaymath}
	\mu \left( \sum\nolimits_{g \in E} f_{g} - f_{g} \circ \lambda_{g}\right) = \sum\nolimits_{g \in E} \mu (f_{g}) - \mu (f_{g} \circ \lambda_{g}) = 0 ,
\end{displaymath} which would clearly constitute a contradiction. Hence, $G$ is not amenable.

($\Longrightarrow$) Consider the set $T \defeq \{ a \in \mathbb{R}G \mid \Vert a \Vert_{1} = 1, \, a\geq 0 \}$. Our assumption (together with the ultrafilter lemma) implies that there exists a set $I$ along with an ultrafilter $\mathscr{U}$ on $I$ and a family $(a_{i})_{i \in I} \in T^{I}$ such that $\lim_{i\to \mathscr{U}} p_{d}(a-ga) = 0$ for every $g \in G$ and every right-invariant continuous pseudo-metric $d$ on $G$. Let us define $\mu \colon \mathrm{RUC}_{b}(G) \to \mathbb{R}$ by \begin{displaymath}
\mu (f) \defeq \lim_{i \to \mathscr{U}} a_{i}(f) \qquad (f \in \mathrm{RUC}_{b}(G)) .
\end{displaymath} It is easy to see that $\mu$ is indeed a well-defined mean on $\mathrm{RUC}_{b}(G)$. In order to prove that $\mu$ is also left-invariant, let $f \in \mathrm{RUC}_{b}(G)$. Without loss of generality, we assume that $\Vert f \Vert_{\infty} \leq 1$. Consider the right-invariant continuous pseudo-metric $d$ on $G$ given by \begin{displaymath}
	d(x,y) \defeq \sup_{g \in G} \vert f(xg) - f(yg) \vert \qquad (x,y \in G) .
\end{displaymath} As $f \colon (G,d) \to [-1,1]$ is $1$-Lipschitz continuous, we conclude that \begin{displaymath}
	\vert \mu (f) - \mu (f \circ \lambda_{g}) \vert = \lim_{i\to \mathscr{U}} \vert (a_{i} - ga_{i})(f) \vert \leq \lim_{i \to \mathscr{U}} p_{d}(a_{i}-ga_{i}) = 0
\end{displaymath} for every $g \in G$. This completes the proof. \end{proof}

In the following theorem we collect some variations of Theorem~\ref{theorem:topological.day}, which may be of rather technical nature, but will turn out useful later on in Section~\ref{section:folner}. The most interesting bit about Theorem~\ref{theorem:better.day} is the equivalence of~(2) and~(3).

\begin{thm}\label{theorem:better.day} Let $G$ be a topological group. The following are equivalent. \begin{enumerate}
	\item[\emph{(1)}] $G$ is amenable.
	\item[\emph{(2)}] For every $\epsilon > 0$, every finite subset $E \subseteq G$, and every right-invariant continuous pseudo-metric $d$ on $G$, there exists $a \in \mathbb{R}G$ with $\Vert a \Vert_{1} = 1$ and $a \geq 0$ such that \begin{displaymath}
		\forall g \in E \, \forall f \colon (G,d) \to [0,1] \text{ $1$-Lipschitz} \colon \quad \vert a(f) - (ga)(f) \vert \leq \epsilon .
	\end{displaymath}
	\item[\emph{(3)}] There exists $\epsilon \in (0,1)$ such that, for every finite subset $E \subseteq G$ and every right-invariant continuous pseudo-metric $d$ on $G$, there exists $a \in \mathbb{R}G$ with $\Vert a \Vert_{1} = 1$ and $a \geq 0$ such that \begin{displaymath}
		\forall g \in E \, \forall f \colon (G,d) \to [0,1] \text{ $1$-Lipschitz} \colon \quad \vert a(f) - (ga)(f) \vert \leq \epsilon .
	\end{displaymath}
	\item[\emph{(4)}] There exists $\epsilon \in (0,1)$ such that, for every finite subset $E \subseteq G$ and every right-invariant continuous pseudo-metric $d$ on $G$, there exists $a \in \mathbb{R}G$ with $\Vert a \Vert_{1} = 1$ and $a \geq 0$ such that \begin{displaymath}
		\forall g \in E \colon \quad p_{d}(a - ga) \leq \epsilon .
	\end{displaymath}
\end{enumerate} \end{thm}

\begin{proof} (1)$\Longrightarrow$(4). This is an immediate consequence of Theorem~\ref{theorem:topological.day}.
	
(4)$\Longrightarrow$(3). This is obvious.
	
(2)$\Longrightarrow$(1). We are going to utilize Theorem~\ref{theorem:topological.day}. To this end, let $\epsilon > 0$, let $E$ be a finite subset of $G$ and let $d$ be a right-invariant continuous pseudo-metric on~$G$. By assumption, there exists $a \in \mathbb{R}G$ with $\Vert a \Vert_{1} = 1$ and $a \geq 0$ such that \begin{displaymath}
	\forall g \in E \, \forall f \colon (G,d) \to [0,1] \text{ $1$-Lipschitz} \colon \quad \vert a(f) - (ga)(f) \vert \leq \tfrac{\epsilon}{2} .
\end{displaymath} We argue that $p_{d}(a - ga) \leq \epsilon$ for all $g \in E$. Let $f \colon (G,d) \to [-1,1]$ be $1$-Lipschitz continuous. Consider the two $1$-Lipschitz functions $f^{+},f^{-} \colon (G,d) \to [0,1]$ given by $f^{+}(x) \defeq f(x) \vee 0$ and $f^{-}(x) \defeq - (f(x) \wedge 0)$ for $x \in G$. It follows that \begin{displaymath}
	\vert a(f) - (ga)(f) \vert \leq \vert a(f^{+}) - (ga)(f^{+}) \vert + \vert a(f^{-}) - (ga)(f^{-}) \vert \leq \epsilon 
\end{displaymath} for every $g \in E$, which completes the argument.
	
(3)$\Longrightarrow$(2). Suppose that~(3) holds for some $\epsilon \in (0,1)$. Let us first prove the following.

\textit{Claim 1.} For every finite subset $E \subseteq G$ and every right-invariant continuous pseudo-metric $d$ on $G$, there exists $a \in \mathbb{R}G$ with $\Vert a \Vert_{1} = 1$ and $a \geq 0$ such that \begin{displaymath}
	\forall g,h \in E \, \forall f \colon (G,d) \to [0,1] \text{ $1$-Lipschitz} \colon \quad \vert (ga)(f) - (ha)(f) \vert \leq \epsilon .
\end{displaymath}

\textit{Proof of Claim 1.} Let $E$ be a finite subset of $G$ and let $d$ be a right-invariant continuous pseudo-metric on~$G$. Consider the right-invariant continuous pseudo-metric $d'$ on $G$ defined by $d'(x,y) \defeq \sup_{g \in E} d(gx,gy)$ for $x,y \in G$. Due to~(3), there exists $a \in \mathbb{R}G$ with $\Vert a \Vert_{1} = 1$ and $a \geq 0$ such that \begin{displaymath}
	\forall g \in E^{-1}E \, \forall f \colon (G,d') \to [0,1] \text{ $1$-Lipschitz} \colon \quad \vert a(f) - (ga)(f) \vert \leq \epsilon .
\end{displaymath} Now, if $g,h \in E$ and $f \colon (G,d) \to [0,1]$ is $1$-Lipschitz continuous, then $f \circ \lambda_{g} \colon (G,d') \to [0,1]$ is $1$-Lipschitz continuous, and hence \begin{displaymath}
	\vert (ga)(f) - (ha)(f) \vert = \vert a(f \circ \lambda_{g}) - (g^{-1}ha)(f \circ \lambda_{g}) \vert \leq \epsilon . \qed
\end{displaymath} 
	
Now we may prove~(2). Consider $\epsilon' > 0$, a finite subset $E \subseteq G$, and a right-invariant continuous pseudo-metric $d$ on $G$. Without loss of generality, assume that $e \in E$. Then there exists some $n \geq 1$ such that $\epsilon^{n} \leq \epsilon'$. Let $d_{i} \defeq \epsilon^{i+1-n}d$ for $i \in \{ 0,\ldots,n-1 \}$. Due to Claim~1, there exists $a_{0} \in \mathbb{R}G$ with $\Vert a_{0} \Vert_{1} = 1$ and $a_{0} \geq 0$ such that \begin{displaymath}
	\forall g,h \in E_{0} \defeq E \, \forall f \colon (G,d_{0}) \to [0,1] \text{ $1$-Lipschitz}\colon \quad \vert (ga_{0})(f) - (ha_{0})(f) \vert \leq \epsilon .
\end{displaymath} Recursively, for each $i \in \{ 1,\ldots,n-1 \}$, we define $E_{i} \defeq E_{i-1} \cup E_{i-1}(\spt a_{i-1})$ and we may apply Claim~1 to choose an element $a_{i} \in \mathbb{R}G$ with $\Vert a_{i} \Vert_{1} = 1$ and $a_{i} \geq 0$ such that \begin{displaymath}
	\forall g,h \in E_{i} \, \forall f \colon (G,d_{i}) \to [0,1] \text{ $1$-Lipschitz}\colon \quad \vert (ga_{i})(f) - (ha_{i})(f) \vert \leq \epsilon .
\end{displaymath} Let $b \defeq a_{0}\cdots a_{n-1}$. Of course, $\Vert b \Vert_{1} = 1$ and $b \geq 0$. We show that $\vert (gb)(f) - (hb)(f) \vert \leq \epsilon'$ for every $1$-Lipschitz continuous function $f \colon (G,d) \to [0,1]$ and all $g,h \in E$. Consider any $1$-Lipschitz continuous function $f \colon (G,d) \to [0,1]$. Then Lemma~\ref{lemma:contractive.functions} asserts the following: for each $i \in \{ 0,\ldots,n-2 \}$, we obtain two $1$-Lipschitz continuous functions $f_{i} \colon (G,d_{i}) \to \mathbb{R}$ and $f^{\ast}_{i} \colon (G,d_{i}) \to [0,1]$ by setting \begin{align*}
	f_{i} (g) &\defeq \epsilon^{i+1-n} \left((R_{a_{i+1}\cdots a_{n-1}}f)(g)-\min_{h \in E_{i+1}} (R_{a_{i+1}\cdots a_{n-1}}f)(h)\right) & (g \in G) , \\
	f^{\ast}_{i} (g) &\defeq (f_{i}(g) \wedge 1) \vee 0 & (g \in G) .
\end{align*} 

\textit{Claim 2.} $f_{i}\vert_{E_{i+1}} = f^{\ast}_{i}\vert_{E_{i+1}}$ for every $i \in \{ 0,\ldots ,n-2 \}$.

\textit{Proof of Claim 2.} Our proof will proceed by downward induction starting at $i = n-2$. Since $d_{n-1} = d$, our choice of $a_{n-1}$ readily implies that \begin{align*}
	0 \leq (R_{a_{n-1}}f)(g)-\min_{h \in E_{n-1}} (R_{a_{n-1}}f)(h) &= \max_{h \in E_{n-1}} (ga_{n-1})(f) - (ha_{n-1})(f) \\
	&\leq \max_{h \in E_{n-1}} p_{d}(ga_{n-1} - ha_{n-1}) \leq \epsilon ,
\end{align*} for all $g \in E_{n-1}$, which means that $f_{n-2}(E_{n-1}) \subseteq [0,1]$ and therefore $f_{n-2}\vert_{E_{n-1}} = f_{n-2}^{\ast}\vert_{E_{n-1}}$. For the inductive step, let us assume that $f_{i+1}\vert_{E_{i+2}} = f^{\ast}_{i+1}\vert_{E_{i+2}}$ for some $i \in \{ 0,\ldots,n-3 \}$. Since $E_{i+1}(\spt a_{i+1}) \subseteq E_{i+2}$, this implies that $(ga_{i+1})(f_{i+1}) = (ga_{i+1})(f_{i+1}^{\ast})$ for all $g \in E_{i+1}$. By our choice of $a_{i+1}$, it follows that \begin{align*}
	0 \ &\leq \ R_{a_{i+1}\cdots a_{n-1}}f(g)-\min_{h \in E_{i+1}} R_{a_{i+1}\cdots a_{n-1}}f(h) \\
		&= \ \max_{h \in E_{i+1}} (ga_{i+1})(R_{a_{i+2}\cdots a_{n-1}}f)-(ha_{i+1})(R_{a_{i+2}\cdots a_{n-1}}f) \\
		&= \ \epsilon^{n-i-2} \left( \max_{h \in E_{i+1}} (ga_{i+1})(f_{i+1}) - (ha_{i+1})(f_{i+1}) \right) \\
		&= \ \epsilon^{n-i-2} \left( \max_{h \in E_{i+1}} (ga_{i+1})(f_{i+1}^{\ast}) - (ha_{i+1})(f_{i+1}^{\ast}) \right) \\
		&\leq \ \epsilon^{n-i-2} \max_{h \in E_{i+1}} p_{d_{i+1}}(ga_{i+1} - ha_{i+1}) \ \leq \ \epsilon^{n-i-1}
\end{align*} whenever $g \in E_{i+1}$. Hence, $f_{i}(E_{i+1}) \subseteq [0,1]$ and therefore $f_{i}\vert_{E_{i+1}} = f_{i}^{\ast}\vert_{E_{i+1}}$. This completes our induction and thus proves Claim~2. $\qed$ 

In particular, $f_{0}\vert_{E_{1}} = f_{0}^{\ast}\vert_{E_{1}}$. As $E_{0}(\spt a_{0}) \subseteq E_{1}$, we conclude that $(ga_{0})(f_{0}) = (ga_{0})(f_{0}^{\ast})$ whenever $g \in E_{0} = E$. Hence, our choice of $a_{0}$ now implies that \begin{align*}
	\vert (gb)(f) - (hb)(f) \vert &= \vert (ga_{0})(R_{a_{1}\cdots a_{n-1}}f)-(ha_{0})(R_{a_{1}\cdots a_{n-1}}f) \vert = \epsilon^{n-1} \vert (ga_{0})(f_{0}) - (ha_{0})(f_{0}) \vert \\
	&= \epsilon^{n-1} \vert (ga_{0})(f_{0}^{\ast}) - (ha_{0})(f_{0}^{\ast}) \vert = \epsilon^{n-1} p_{d_{0}}(ga_{0} - ha_{0}) \leq \epsilon^{n} \leq \epsilon'
\end{align*} for all $g,h \in E$, and we are done. \end{proof}

\section{F\o lner conditions for topological groups}\label{section:folner}

Our next objective is to establish a topological version of the following well-known amenability criterion for discrete groups due to F\o lner~\cite{folner}. The reader is referred to~\cite{namioka} for a very short and lucid proof of F\o lner's theorem.

\begin{thm}[\cite{folner}]\label{theorem:folner} A group $G$ is amenable if and only if, for every $\theta \in (0,1)$ and every finite subset $E \subseteq G$, there exists a finite non-empty subset $F \subseteq G$ such that \begin{displaymath}
	\forall g \in E \colon \quad \vert F \cap gF \vert \geq \theta \vert F \vert .
\end{displaymath} \end{thm}

Our topological version of F\o lner's theorem will be in terms of matchings with respect to identity neighborhoods in topological groups (Theorem~\ref{theorem:topological.folner}), and its proof will rely on the results of Section~\ref{section:day} as well as a quantitative version of Hall's marriage theorem (Theorem~\ref{theorem:hall}).

For a start, let us clarify some terminology and notation. Let $\mathscr{B} = (X,Y,R)$ be a \emph{bipartite graph}, i.e., a triple consisting of two finite sets $X$ and $Y$ and a relation $R \subseteq X \times Y$. If $S \subseteq X$, then we define $N_{\mathscr{B}}(S) \defeq \{ y \in Y \mid \exists x \in S \colon (x,y) \in R \}$. A \emph{matching} in $\mathscr{B}$ is an injective map $\phi \colon D \to Y$ such that $D \subseteq X$ and $(x,\phi(x)) \in R$ for all $x \in D$. A matching $\phi$ in $\mathscr{B}$ is said to be \emph{perfect} if $\dom (\phi) = X$. Furthermore, the \emph{matching number} of $\mathscr{B}$ is defined to be \begin{displaymath}
	\mu (\mathscr{B}) \defeq \sup \{ |\dom \phi | \mid \phi \textnormal{ matching in } \mathscr{B} \} .
\end{displaymath} For convenience, we restate Hall's well-known matching theorem.

\begin{thm}[\cite{Hall35}, \cite{Ore}]\label{theorem:hall} If $\mathscr{B} = (X,Y,R)$ is a bipartite graph, then \begin{displaymath}
	\mu (\mathscr{B}) = |X| - \sup \{ |S| - |N_{\mathscr{B}}(S)| \mid S \subseteq X \} .
\end{displaymath} \end{thm}

\begin{cor}[\cite{Hall35}]\label{corollary:hall} A bipartite graph $\mathscr{B} = (X,Y,R)$ admits a perfect matching if and only if $|S| \leq |N_{\mathscr{B}}(S)|$ for every subset $S \subseteq X$. \end{cor}

In what follows, we will have a closer look at certain matchings in topological groups.

\begin{definition} Let $G$ be a topological group, $E,F \in \mathscr{F}(G)$, and $U \in \mathscr{U}_{e}(G)$. We define the bipartite graph $\mathscr{B} (E,F,U) \defeq (E,F,R(E,F,U))$ with the relation given by \begin{displaymath}
	R(E,F,U) \defeq \{ (x,y) \in E \times F \mid yx^{-1} \in U \} .
\end{displaymath} Furthermore, we abbreviate $\mu (E,F,U) \defeq \mu (\mathscr{B}(E,F,U))$. \end{definition}

Clearly, if $E$ and $F$ are finite subsets of a topological group $G$, then \begin{displaymath}
	\mu (E,F,U) \leq \mu (E,F,V)
\end{displaymath} for any pair of identity neighborhoods $U,V \in \mathscr{U}_{e}(G)$ with $U \subseteq V$. Moreover, if $G$ is a discrete group, then $\mu (E,F,\{ e \}) = \vert E \cap F \vert$ for any two finite subsets $E ,F \subseteq G$. Hence, the following theorem may be regarded as a topological version of F\o lner's amenability criterion.

\begin{thm}\label{theorem:topological.folner} Let $G$ be a Hausdorff topological group. The following are equivalent. \begin{enumerate}
	\item[\emph{(1)}] $G$ is amenable.
	\item[\emph{(2)}] For every $\theta \in (0,1)$, every finite subset $E \subseteq G$, and every $U \in \mathscr{U}_{e}(G)$, there exists a finite non-empty subset $F \subseteq G$ such that \begin{displaymath}
		\forall g \in E \colon \quad \mu (F,gF,U) \geq \theta \vert F \vert .
	\end{displaymath}
	\item[\emph{(3)}] There exists $\theta \in (0,1)$ such that, for every finite subset $E \subseteq G$ and every $U \in \mathscr{U}_{e}(G)$, there exists a finite non-empty subset $F \subseteq G$ such that \begin{displaymath}
		\forall g \in E \colon \quad \mu (F,gF,U) \geq \theta \vert F \vert .
	\end{displaymath}
\end{enumerate} \end{thm}

\begin{proof} (2)$\Longrightarrow$(3). This is obvious.
	
(1)$\Longrightarrow$(2). Since the discrete case is covered by Theorem~\ref{theorem:folner} already, we will insist on $G$ being non-discrete. Let $\theta \in (0,1)$, $E \in \mathscr{F}(G)$ and $U \in \mathscr{U}_{e}(G)$. Without loss of generality, we may assume that $e \in E$. Due to Urysohn's lemma for uniform spaces, there exists a right-uniformly continuous function $f \colon G \to [0,1]$ such that $\spt (f) \subseteq U$ and $f(e) = 1$. Consider the continuous right-invariant pseudo-metrics $d$ and $d'$ on $G$ given by \begin{displaymath}
	d(x,y) \defeq \sup_{g \in G} \vert f(xg) - f(yg) \vert , \qquad d'(x,y) \defeq \sup_{g \in E} d(gx,gy) \qquad (x,y \in G) .
\end{displaymath} By Theorem~\ref{theorem:topological.day}, there is $a \in \mathbb{R}G$ with $\Vert a \Vert_{1} = 1$, $a \geq 0$, and $p_{d}(a-ga) \leq \frac{1-\theta}{3}$ for all~$g \in E$. Since $G$ is not discrete, it constitutes a perfect Hausdorff space. Due to Lemma~\ref{lemma:density}, there thus exists a finite non-empty subset $F \subseteq G$ such that $p_{d'}(a - \delta_{F}) \leq \frac{1-\theta}{3}$. It follows that \begin{displaymath}
	p_{d}(\delta_{F} - \delta_{gF}) \leq p_{d}(\delta_{F} - a) + p_{d}(a-ga) + p_{d}(ga-g\delta_{F}) \leq p_{d}(a-ga) + 2p_{d'}(a-\delta_{F}) \leq 1-\theta
\end{displaymath} for all $g \in E$. We show that $\mu (F,gF,U) \geq \theta |F|$ for every $g \in E$. So, let $g \in E$ and consider the bipartite graph $\mathscr{B} \defeq \mathscr{B}(F,gF,U)$. Let $S \subseteq F$. Since $d$ is right-invariant, $f \circ \rho_{h} \colon (G,d) \to [0,1]$ is $1$-Lipschitz continuous for every $h \in G$. Hence, the function $f_{S} \colon (G,d) \to [0,1]$ given by \begin{displaymath}
	f_{S}(x) \defeq \bigvee_{y \in S} f(xy^{-1}) \qquad (x \in G)
\end{displaymath} is $1$-Lipschitz continuous, too. Therefore, $\vert \delta_{F}(f_{S}) - \delta_{gF}(f_{S}) \vert \leq 1 - \theta$. Also note that $f_{S}\vert_{S} = 1$. As $\spt (f) \subseteq U$, it now follows that \begin{align*}
	|S| &\leq \vert F \vert \delta_{F}(f_{S}) \leq (1 - \theta)\vert F \vert + \vert F \vert \delta_{gF}(f_{S}) = (1 - \theta)\vert F \vert + \sum_{x \in gF} f_{S}(x) \\
	&= (1 - \theta)\vert F \vert + \sum_{x \in N_{\mathscr{B}}(S)} f_{S}(x) \leq (1-\theta) \vert F \vert + \vert N_{\mathscr{B}}(S) \vert ,
\end{align*} that is, $|S| - \left|N_{\mathscr{B}}(S)\right| \leq (1-\theta ) \vert F\vert$. Applying Theorem~\ref{theorem:hall}, we arrive at \begin{align*}
	\frac{\mu (F,gF,U)}{|F|} &= \frac{|F|-\sup_{S \subseteq F} \left( |S| - \left| N_{\mathscr{B}}(S) \right|\right)}{|F|} \geq \frac{|F|-(1-\theta )|F|}{|F|} = \theta .
\end{align*}

(3)$\Longrightarrow$(1). We are going to prove that if (2) holds for some $\theta \in (0,1)$, then the third condition of Theorem~\ref{theorem:better.day} is satisfied for $\epsilon \defeq 1 - \frac{\theta}{2}$. Consider a finite subset $E \subseteq G$ and a right-invariant continuous pseudo-metric $d$ on $G$. Then $U \defeq \{ x \in G \mid d(x,e) \leq \tfrac{1}{2} \}$ is an identity neighborhood on $G$. According to~(2), there exists a finite non-empty subset $F \subseteq G$ such that $\mu (F,gF,U) \geq \theta \vert F \vert$ for all $g \in E$. We show that $p_{d}(\delta_{F}-g\delta_{F}) \leq \epsilon$ for every $g \in E$. To this end, let $g \in E$ and consider an injective map $\phi \colon D \to gF$ with $D \subseteq F$, $|D| = \mu (F,gF,U)$, and $\phi (x) \in Ux$ for all $x \in D$. Let $\bar{\phi} \colon F \to gF$ be any bijection with $\bar{\phi}\vert_{D} = \phi$. Now, if $f \colon (G,d) \to [0,1]$ is $1$-Lipschitz continuous, then \begin{align*}
	|\delta_{F}(f) - (g\delta_{F})(f)| &= \frac{1}{|F|}\left|\sum_{x \in F} f(x) - \sum_{x \in F} f(gx)\right| \\
		&= \frac{1}{|F|}\left|\sum_{x \in D} (f(x)-f(\bar{\phi} (x))) + \sum_{x \in F\setminus D} (f(x) - f(\bar{\phi} (x))\right| \\
		&\leq \frac{1}{|F|} \sum_{x \in D} |f(x)-f(\phi (x))| + \frac{1}{|F|} \sum_{x \in F\setminus D} \vert f(x) - f(\bar{\phi}(x)) \vert \\
		&\leq \frac{\vert D \vert}{2\vert F \vert} + \frac{\vert F\vert - \vert D \vert}{|F|} \\
		&\leq 1 - \frac{\theta}{2} = \epsilon .\qedhere
\end{align*} \end{proof}

\begin{remark}\label{remark:topological.folner} Let $G$ be a Hausdorff topological group. According to Theorem~\ref{theorem:topological.folner}, $G$ is amenable if and only if there exists a family of finite non-empty subsets $(F_{i})_{i \in I}$ of $G$ along with a (non-principal) ultrafilter $\mathscr{U}$ on $I$ such that \begin{displaymath}
	\forall g\in G \, \forall U \in \mathscr{U}_{e}(G) \colon \quad \lim_{i \to \mathscr{U}} \frac{\mu (F_{i},gF_{i},U)}{\vert F_{i} \vert} = 1 ,
\end{displaymath} which by the estimate in the proof of the implication (3)$\Longrightarrow$(1) of Theorem~\ref{theorem:topological.folner} implies that \begin{displaymath}
	\forall g \in G \, \forall f \in \mathrm{RUC}_{b}(G) \colon \quad \lim_{i \to \mathscr{U}} \delta_{F_{i}}(f-f \circ \lambda_{g}) = 0 ,
\end{displaymath} which in turn implies amenability of $G$ by the argument used to prove the implication ($\Longleftarrow$) of Theorem~\ref{theorem:topological.day}. This provides us with another (a priori weaker) characterization of amenability, which had been observed in~\cite{SchneiderThom} already. \end{remark}

Let us mention the following rather direct consequence of Theorem~\ref{theorem:topological.folner} concerning matchings with respect to pairs of translates.

\begin{cor} A Hausdorff topological group $G$ is amenable if and only if, for every $\theta \in (0,1)$, every finite subset $E \subseteq G$ and every $U \in \mathscr{U}_{e}(G)$, there exists a finite non-empty subset $F \subseteq G$ such that \begin{displaymath}
	\forall g,h \in E \colon \quad \mu (gF,hF,U) \geq \theta \vert F \vert .
\end{displaymath} \end{cor}

\begin{proof} ($\Longleftarrow$) This is due to Theorem~\ref{theorem:topological.folner}.

($\Longrightarrow$) Assume that $G$ is amenable. Let $\theta \in (0,1)$, $U \in \mathscr{U}_{e}(G)$, and consider a finite subset $E \subseteq G$. Then $V \defeq \bigcap_{g \in E} g^{-1}Ug$ is an identity neighborhood in $G$, too. Hence, Theorem~\ref{theorem:topological.folner} asserts the existence of a finite non-empty subset $F \subseteq G$ such that $\mu (F,gF,V) \geq \theta \vert F \vert$ for all $g \in E^{-1}E$. We prove that $\mu (gF,hF,U) \geq \theta \vert F \vert$ for all $g,h \in E$. To this end, let $g,h \in E$. Consider an injective map $\phi \colon D \to g^{-1}hF$ such that $D \subseteq F$, $|D| = \mu (F,g^{-1}hF,V)$, and $\phi (x) \in Vx$ for all $x \in D$. Define $D' \defeq gD$. Clearly, \begin{displaymath}
	\psi \colon D' \to hF, \quad x \mapsto g\phi (g^{-1}x) 
\end{displaymath} is injective, and $\psi (x) = g\phi (g^{-1}x) \in gVg^{-1}x \subseteq Ux$ for every $x \in D'$. Hence, \begin{displaymath}
	\mu (gF,hF,U) \geq \vert D' \vert = \vert D \vert = \mu (F,g^{-1}hF,V) \geq \theta \vert F \vert . \qedhere
\end{displaymath} \end{proof}

We conclude this section with a first application of Theorem~\ref{theorem:topological.folner}. It was asked in~\cite{BarrosoMbomboPestov} to what extent amenability of topological groups can be characterized by the existence of approximate fixed points -- or equivalently, almost invariant vectors -- for their affine continuous actions on bounded (but not necessarily compact) convex subsets of locally convex spaces. An action of a group $G$ by affine homeomorphisms on a convex subset $C$ of some topological vector space is said to have \emph{approximate fixed points}~\cite{BarrosoMbomboPestov} if there is a net $(x_{\iota})_{\iota \in I}$ in $C$ such that $x_{\iota} - gx_{\iota} \to 0$ for every $g \in G$. We resolve the question raised in~\cite{BarrosoMbomboPestov} as follows, where an action of a topological group $G$ by isomorphisms on a uniform space $X$ will be called \emph{bounded}~\cite{pestovbook} if, for every entourage $\alpha$ of $X$, there exists $U \in \mathscr{U}_{e}(G)$ such that \begin{displaymath}
	\forall g \in U \, \forall x \in X \colon \quad (x,gx) \in \alpha . 
\end{displaymath} 

\begin{cor} Let $G$ be a Hausdorff topological group. The following are equivalent. \begin{enumerate}
	\item[$(1)$] $G$ is amenable.
	\item[$(2)$] Every bounded action of $G$ by affine homeomorphisms on a bounded convex subset of a locally convex space has approximate fixed points.
\end{enumerate} \end{cor}
	
\begin{proof} The implication from~(2) to~(1) is an immediate consequence of Rickert's fixed point theorem~\cite[Theorem~4.2]{rickert} and the fact that continuous actions of topological groups on compact spaces are bounded. To prove the converse, consider any bounded action of $G$ by affine homeomorphisms on a bounded convex subset $C$ of a locally convex space $X$. Fix some $x \in C$. By Theorem~\ref{theorem:topological.folner}, there is a net $(F_{\iota})_{\iota \in I}$ of finite non-empty subsets of $G$ such that \begin{displaymath}
	\forall g \in G \, \forall U \in \mathscr{U}_{e}(G) \colon \quad \frac{\mu (F_{\iota},gF_{\iota},U)}{\vert F_{\iota} \vert} \to 1 .
\end{displaymath} For each $\iota \in I$, let us consider \begin{displaymath}
	x_{\iota} \defeq \frac{1}{\vert F_{\iota} \vert}\sum_{h \in F_{\iota}} hx \in C .
\end{displaymath} We claim that $x_{\iota} - gx_{\iota} \to 0$ for every $g \in G$. To see this, let $g \in G$ and let $N$ be a zero neighborhood in $X$. Since $X$ is locally convex, there is a balanced, convex zero neighborhood $V$ in $X$ such that $V + V \subseteq N$. As $C$ is bounded, there exists $\alpha \in \mathbb{K}$ such that $C-C \subseteq \alpha V$. Since $G$ acts on $C$ in a bounded manner, we find some $U \in \mathscr{U}_{e}(G)$ so that $hx \in V+x$ for all $x \in C$ and $h \in U$. By our hypothesis about $(F_{\iota})_{\iota \in I}$, there exists $\iota_{0} \in I$ such that \begin{displaymath}
	\forall \iota \in I , \, \iota \geq \iota_{0} \colon \quad \frac{\mu (F_{\iota},gF_{\iota},U)}{\vert F_{\iota} \vert} \geq 1 - \vert \alpha \vert^{-1} .
\end{displaymath} We show that $x_{\iota} - gx_{\iota} \in V$ for all $\iota \in I$ with $\iota \geq \iota_{0}$. So, let $\iota \in I$, $\iota \geq \iota_{0}$. Due to the above, there exists a subset $D_{\iota} \subseteq F_{\iota}$ along with a bijection $\phi_{\iota} \colon F_{\iota} \to gF_{\iota}$ such that $\phi_{\iota}(h) \in Uh$ for each $h \in D_{\iota}$ and moreover \begin{displaymath}
	\frac{\vert D_{\iota} \vert}{\vert F_{\iota} \vert} \geq 1 - \vert \alpha \vert^{-1} , \qquad \text{i.e.} \qquad \frac{\vert F_{\iota} \setminus D_{\iota} \vert}{\vert F_{\iota} \vert} \vert \alpha \vert \leq 1 .
\end{displaymath} It follows that \begin{align*}
	x_{\iota} & - gx_{\iota} \, = \, \frac{1}{|F_{\iota}|} \left( \sum_{h \in F_{\iota}} hx - \sum_{h \in F_{\iota}} ghx \right) \\
		&= \, \frac{1}{|F_{\iota}|}\left( \sum_{h \in D_{\iota}} hx-\phi_{\iota}(h)x + \sum_{h \in F_{\iota}\setminus D_{\iota}} hx - \phi_{\iota} (h) x \right) \\
		&= \, \frac{\vert D_{\iota} \vert}{\vert F_{\iota} \vert} \left( \frac{1}{|D_{\iota}|} \sum_{h \in D_{\iota}} hx-\phi_{\iota}(h)x \right) + \frac{\vert F_{\iota} \setminus D_{\iota} \vert}{\vert F_{\iota} \vert} \left( \frac{1}{\vert F_{\iota} \setminus D_{\iota} \vert} \sum_{h \in F_{\iota}\setminus D_{\iota}} hx - \phi_{\iota} (h) x \right) \\
		&\in \, \frac{\vert D_{\iota} \vert}{\vert F_{\iota} \vert}V + \frac{\vert F_{\iota} \setminus D_{\iota} \vert}{\vert F_{\iota} \vert} (C-C) \, \subseteq \, V + \frac{\vert F_{\iota} \setminus D_{\iota} \vert}{\vert F_{\iota} \vert} \alpha V \, \subseteq \, V + V \, \subseteq \, N ,
\end{align*} which completes the argument. \end{proof}

\section{Perturbed translations and the topological von Neumann problem}\label{section:whyte}

In this section we give an application of the results obtained in Section~\ref{section:folner}: we prove that a topological group is amenable if and only if its left translation action is uniformly approximable by amenable actions (Theorem~\ref{theorem:approximating.actions}). This result in turn yields a characterization of non-amenability in terms of paradoxical decompositions (Remark~\ref{remark:paradoxical}), provides a topological solution to the von Neumann problem (Corollary~\ref{corollary:von.neumann}), and even allows for measurable variations concerning Polish groups (Corollaries~\ref{corollary:baire.paradoxical} and~\ref{corollary:baire.free}).

We start off by clarifying some terminology. We say that an action $(G,X)$ of a group $G$ on a non-empty set $X$ is \emph{amenable} if $\ell^{\infty}(X)$ admits a $G$-invariant mean, i.e., there is a positive linear map $\mu \colon \ell^{\infty}(X) \to \mathbb{R}$ with $\mu (\mathbf{1}) = 1$ and $\mu (f \circ g) = \mu (f)$ for all $f \in \ell^{\infty}(X)$ and $g \in G$. By Tarski's alternative (cf.~\cite{ParadoxicalDecompositions}), a group action $(G,X)$ is non-amenable if and only if $X$ admits a \emph{$G$-paradoxical decomposition}, that is, there exist $g_{1},\ldots,g_{m},h_{1},\ldots,h_{n} \in G$ and partitions $X = \bigcupdot_{i=1}^{m} A_{i} = \bigcupdot_{j=1}^{n} B_{j}$ such that $X = \bigcupdot_{i=1}^{m} g_{i}A_{i} \cupdot \bigcupdot_{j=1}^{n} h_{j}B_{j}$. Furthermore, F\o lner's classical amenability criterion~\cite{folner} generalizes to group actions as follows.

\begin{thm}[Rosenblatt~\cite{rosenblatt}]\label{theorem:amenable.actions} A group action $(G,X)$ is amenable if and only if, for every $\theta > 1$ and every finite subset $E \subseteq G$, there exists a finite non-empty subset $F \subseteq X$ such that $\vert EF \vert \leq \theta \vert F \vert$. \end{thm}

\begin{remark}\label{remark:amenable.actions} Let $G$ be a group. If $S \subseteq G$ generates $G$, then Theorem~\ref{theorem:amenable.actions} (together with an obvious estimate) asserts the following: an action of $G$ on a set $X$ is amenable if and only if, for every $\theta > 1$ and every finite subset $E \subseteq S$, there exists a finite non-empty subset $F \subseteq X$ such that $\vert EF \vert \leq \theta \vert F \vert$. \end{remark}

For convenience, let us also introduce some very few bits of additional notation. Let $I$ and $X$ be sets. For a mapping $\alpha \colon I \to \mathrm{Sym}(X)$, we will denote by $\Gamma (\alpha)$ the subgroup of $\mathrm{Sym} (X)$ generated by the subset $\{ \alpha (i) \mid i \in I \}$. Given a family of maps $\alpha_{j} \colon I \to \mathrm{Sym}(X)$ $(j \in J)$, we define $\bigoplus_{j \in J} \alpha_{j} \colon I \to \mathrm{Sym}(X \times J)$ by setting \begin{displaymath}
	\left( \bigoplus\nolimits_{j \in J} \alpha_{j} \right)(i)(x,j) \defeq (\alpha_{j}(i)(x),j) \qquad (i \in I, \, j \in J, \, x \in X) .
\end{displaymath}

Now we turn towards topological groups and their perturbed left translations. Let $G$ be a topological group. We will consider the induced topology of uniform convergence on $\mathrm{Sym}(G)^{G}$ viewed as a subspace of $(G^{G})^{G} \cong G^{G \times G}$, where $G$ is carrying its right uniformity. It is not difficult to see that the sets of the form \begin{displaymath}
	\left. \mathscr{N}_{\alpha}(U) \defeq \left\{ \beta \in \mathrm{Sym}(G)^{G} \, \right\vert \forall g,h \in G \colon \, \beta (g)(h) \in U\alpha (g)(h) \right\} \qquad (U \in \mathscr{U}_{e}(G) )
\end{displaymath} constitute a neighborhood basis at any $\alpha \in \mathrm{Sym}(G)^{G}$ with respect to the considered topology. Being particularly interested in the map $\lambda \colon G \to \mathrm{Sym}(G), \, g \mapsto \lambda_{g}$, we let $\mathscr{N}_{G}(U) \defeq \mathscr{N}_{\lambda}(U)$ for every $U \in \mathscr{U}_{e}(G)$. Let $U \in \mathscr{U}_{e}(G)$. Given~$\alpha \in \mathscr{N}_{G}(U)$, the group $\Gamma (\alpha)$ may be viewed as an approximation of $G$ up to the neighborhood $U$. Extending this idea, one may think of \begin{displaymath}
	\Gamma_{U}(G) \, \defeq \, \Gamma \! \left(\bigoplus\nolimits_{\alpha \in \mathscr{N}_{G}(U)} \alpha \right) \, \leq \, \mathrm{Sym}(G \times \mathscr{N}_{G}(U))
\end{displaymath} as the $U$-perturbation of $G$. And indeed, as the following result reveals, those approximations do reflect the amenability of $G$ in a fairly natural manner.

\begin{thm}\label{theorem:approximating.actions} Let $G$ be a Hausdorff topological group. The following are equivalent. \begin{enumerate}
	\item[$(1)$] $G$ is amenable.
	\item[$(2)$] For every $U \in \mathscr{U}_{e}(G)$, there exists $\alpha \in \mathscr{N}_{G}(U)$ such that the action of $\Gamma (\alpha)$ on $G$ is amenable.
	\item[$(3)$] For every $U \in \mathscr{U}_{e}(G)$, the action of $\Gamma_{U}(G)$ on $G \times \mathscr{N}_{G}(U)$ is amenable.
\end{enumerate} \end{thm}

Our Theorem~\ref{theorem:approximating.actions} establishes a link between amenability of topological groups and amenability of group actions on sets: the equivalence of~(1) and~(2) means that a topological group $G$ is amenable if and only if the action of $G$ on itself by left translations can be uniformly approximated by amenable actions on the set $G$. Evidently, condition~(3) is a priori weaker than~(2), and therefore $\neg$(3) provides a comparably strong manifestation of non-amenability, as we are going to illustrate by the subsequent remark. Let us agree on the following notation: for a set $X$, let $F_{X}$ denote the free group over the base set $X$, and for a map $\phi \colon X \to G$ into a group $G$, denote by $\phi^{\ast} \colon F_{X} \to G$ the unique homomorphism with $\phi^{\ast}\vert_{X} = \phi$.

\begin{remark}\label{remark:paradoxical} Let $G$ be a non-amenable Hausdorff topological group. Appealing to condition~(3) in Theorem~\ref{theorem:approximating.actions}, we find $U \in \mathscr{U}_{e}(G)$ such that $G \times \mathscr{N}_{U}(G)$ admits a $\Gamma_{U}(G)$-paradoxical decomposition, i.e., there exist $g_{1},\ldots ,g_{m},h_{1},\ldots ,h_{n} \in F_{G}$ such that, for every $\alpha \in \mathscr{N}_{G}(U)$, there exist partitions $G = \bigcupdot_{i=1}^{m} A_{i} = \bigcupdot_{j=1}^{n} B_{j}$ with \begin{displaymath}
	G = \bigcupdot_{i=1}^{m} \alpha^{\ast}(g_{i})(A_{i}) \cupdot \bigcupdot_{j=1}^{n} \alpha^{\ast}(h_{j})(B_{j}) .
\end{displaymath} This is a uniform version of the statement that $G$ admits a $\Gamma (\alpha)$-paradoxical decomposition for every $\alpha \in \mathscr{N}_{G}(U)$, which would be the negation of condition~(2) in Theorem~\ref{theorem:approximating.actions}. It seems natural to study the associated \emph{Tarski numbers} (cf.~\cite{ParadoxicalDecompositions}), i.e., the minimal numbers of pieces in the occurring paradoxical decompositions. Considering for each $U \in \mathscr{U}_{e}(G)$ the Tarski number $\tau_{U}(G)$ of the $\Gamma_{U}(G)$-action on $G \times \mathscr{N}_{G}(U)$, one may define the \emph{topological Tarski number} of $G$ to be $\tau_{\mathrm{top}}(G) \defeq \inf \{ \tau_{U}(G) \mid U \in \mathscr{U}_{e}(G) \}$. Clearly, $\tau_{U}(G) \leq \tau_{V}(G)$~for any two $U,V \in \mathscr{U}_{e}(G)$ with $U \subseteq V$. Furthermore, the topological Tarski number of a non-amenable topological group is bounded from below by the Tarski number of the underlying (necessarily non-amenable) discrete group. Of course, the two quantities coincide for discrete non-amenable groups. \end{remark}

Let us briefly outline the proof of Theorem~\ref{theorem:approximating.actions}. Whereas the equivalence of~(1) and~(3) can be deduced merely from our Theorem~\ref{theorem:topological.folner}, and~(3) is an immediate consequence of~(2), the fact that~(1) implies~(2) requires a more slightly more complicated argument relying on the well-ordering theorem in the form of Lemma~\ref{lemma:well.ordering} and involving Zorn's lemma along with Lemma~\ref{lemma:nice.folner.sets} for a maximality argument. Our proof of Lemma~\ref{lemma:nice.folner.sets} will appeal to Theorem~\ref{theorem:topological.folner} as well as the following useful observation.

\begin{lem}\label{lemma:approximation.by.moving.sets} Let $G$ be a non-discrete Hausdorff topological group. If $U$ is an identity neighborhood in $G$ and $E,F \in \mathscr{F}(G)$, then there exists an injective map $\phi \colon F \to G$ such that $\phi (x) \in Ux$ for every $x \in F$ and $\phi (F) \cap g\phi (F) = \emptyset$ for every $g \in E\setminus \{ e \}$. \end{lem}

\begin{proof} Denote by $\mathscr{Z}$ the set of all pairs $(D,\phi)$ consisting of a subset $D \subseteq F$ and an injective map $\phi \colon D \to G$ such that $\phi (x) \in Ux$ for all $x \in D$ and $\phi (D) \cap g\phi (D) = \emptyset$ for every $g \in E\setminus \{ e \}$. Let $(D,\phi) \in \mathscr{Z}$ such that $|D| = \sup \{ |D'| \mid (D',\psi) \in \mathscr{Z} \}$. We claim that $|D| = |F|$. For contraposition, assume that $|D| < |F|$. Then there exists $x \in F\setminus D$. Since $G$ is a perfect Hausdorff space, every open non-empty subset of $G$ is infinite. Consequently, there exists some \begin{displaymath}
	y \in Ux \setminus \left( \phi (D) \cup \bigcup\nolimits_{g \in E} g\phi (D) \cup g^{-1}\phi (D) \right) .
\end{displaymath} We define $D' \defeq D \cup \{ x \}$ and $\psi \colon D' \to F$ such that $\psi|_{D} = \phi$ and $\psi(x) = y$. We observe that $(D',\psi)$ is a member of $\mathscr{Z}$. Since $|D| < |D'|$, this clearly contradicts our hypothesis. Hence, $|D| = |F|$. This finishes the proof. \end{proof}

\begin{lem}\label{lemma:nice.folner.sets} A non-discrete Hausdorff topological group $G$ is amenable if and only if the following holds: for every $\theta \in (0,1)$, every finite subset $E \subseteq G$, and every $U \in \mathscr{U}_{e}(G)$, there exist finite non-empty subsets $D \subseteq F \subseteq G$ and injections $\phi_{g} \colon D \to gF$ $(g \in E)$ such that \begin{enumerate}
	\item[$(1)$] $\vert D \vert \geq \theta \vert F \vert$,
	\item[$(2)$] $F \cap gF = \emptyset$ for every $g \in E\setminus \{ e \}$,
	\item[$(3)$] $\phi_{g}(x) \in Ux$ for all $g \in E$ and $x \in D$.
\end{enumerate} \end{lem}
	
\begin{proof} ($\Longleftarrow$) This is an immediate consequence of Theorem~\ref{theorem:topological.folner}.
		
($\Longrightarrow$) Let $\theta \in (0,1)$, $E \subseteq \mathscr{F}(G)$, and $U \in \mathscr{U}_{e}(G)$. Without loss of generality, suppose that $e \in E = E^{-1}$. Choose $V \in \mathscr{U}_{e}(G)$ with $V^{-1} = V$ and $V^{3} \subseteq U$, and let $W \defeq \bigcap_{g \in E} g^{-1}Vg$. By Theorem~\ref{theorem:topological.folner}, there exists a finite non-empty subset $F_{0} \subseteq G$ such that \begin{displaymath}
	\forall g \in E \colon \quad \mu (F_{0},gF_{0},V) \geq \left(1-\tfrac{1-\theta}{\vert E \vert}\right)\vert F_{0} \vert .
\end{displaymath} By Lemma~\ref{lemma:approximation.by.moving.sets}, there is an injection $\alpha \colon F_{0} \to G$ with $\alpha (x) \in Wx$ for all $x \in F_{0}$. For each $g \in E$, choose a subset $D_{g} \subseteq F_{0}$ and an injection $\psi_{g} \colon D_{g} \to gF_{0}$ such that $\vert D_{g} \vert = \mu (F_{0},gF_{0},V)$ and $\psi_{g}(x) \in Vx$ for all $x \in D_{g}$. Define $F \defeq \alpha (F_{0})$ and $D \defeq \bigcap_{g \in E} \alpha (D_{g})$. Clearly, $F \cap gF = \emptyset$ for every $g \in E\setminus \{ e \}$. Furthermore, note that \begin{align*}
	\vert D \vert &= \vert F \vert - \vert F\setminus D \vert = \vert F_{0} \vert - \left\vert \bigcup\nolimits_{g \in E} \alpha(F_{0}\setminus D_{g}) \right\vert \geq \vert F_{0} \vert - \sum\nolimits_{g \in E} \vert F_{0}\setminus D_{g} \vert \\
	&\geq \vert F_{0}\vert - \vert E \vert \left( \vert F_{0} \vert - \left(1-\tfrac{1-\theta}{\vert E \vert}\right)\vert F_{0} \vert \right) = \theta \vert F_{0} \vert = \theta \vert F \vert .
\end{align*} Finally, for each $g \in E$, we consider the injection $\phi_{g} \colon D \to gF$ given by \begin{displaymath}
	\phi_{g} (x) \defeq g\alpha(g^{-1}\psi_{g}(\alpha^{-1}(x))) \qquad (x \in D) ,
\end{displaymath} and we observe that \begin{displaymath}
	\phi_{g}(x) \in gWg^{-1}\psi_{g}(\alpha^{-1}(x)) \subseteq V\psi_{g}(\alpha^{-1}(x)) \subseteq V^{2}\alpha^{-1}(x) \subseteq V^{3}x \subseteq Ux
\end{displaymath} for all $x \in D$. This completes the argument. \end{proof}
	
For convenience, we recall the following well-known variation of the well-ordering theorem.
	
\begin{lem}\label{lemma:well.ordering} For every set $X$, there exists a well-ordering $<$ on $X$ such that \begin{displaymath}
	\forall x \in X \colon \quad \vert \{ y \in X \mid y < x \} \vert < \vert X \vert .
\end{displaymath} \end{lem}
	
\begin{proof} Let $X$ be a set. By the well-ordering theorem, there is a well-ordering $<$ on $X$. If $\vert \{ y \in X \mid y < x \} \vert < \vert X \vert$ for all $x \in X$, then we are done. Otherwise, \begin{displaymath}
	Y \defeq \{ x \in X \mid \vert \{ y \in X \mid y < x \} \vert = \vert X \vert \} \ne \emptyset ,
\end{displaymath} and thus $(Y,{<})$ admits a smallest element $z \in Y$. Choose a bijection $f \colon X \to \{ x \in X \mid x < z \}$ and define a relation $\prec$ on $X$ by \begin{displaymath}
	x \prec x' \ :\Longleftrightarrow \ f(x) < f(x') \qquad (x,x' \in X) .
\end{displaymath} Clearly, $(X,{\prec})$ is a well-ordered set. To prove the additional property, consider any $x \in X$. Since $f(x) < z$ and $z$ is the smallest element of $(Y,{<})$, we have $\vert \{ y \in X \mid y < f(x) \} \vert < \vert X \vert$. Moreover, as $f(\{ y \in X \mid y \prec x \}) = \{ y \in X \mid y < f(x) \}$, it follows that \begin{displaymath}
	\vert \{ y \in X \mid y \prec x \} \vert = \vert \{ y \in X \mid y < f(x) \} \vert < \vert X \vert .\qedhere
\end{displaymath} \end{proof}
	
Now everything is in place for the proof of Theorem~\ref{theorem:approximating.actions}.
	
\begin{proof}[Proof of Theorem~\ref{theorem:approximating.actions}] Note that the discrete case is trivial, which is why we will assume that $G$ is not discrete.
	
(1)$\Longrightarrow$(2). Due to being a non-discrete Hausdorff space, $G$ is infinite. Thus, $\vert I \vert = \vert G \vert$ for $I \defeq \{ (E,n) \in \mathscr{F}(G) \times (\mathbb{N}\setminus \{ 0 \} ) \mid e \in E \}$. By Lemma~\ref{lemma:well.ordering}, there exists a well-ordering $<$ on $I$ such that $\vert \{ (E',n') \in I \mid (E',n') < (E,n) \} \vert < \vert I \vert$ for every $(E,n) \in I$. Let $U \in \mathscr{U}_{e}(G)$. Without loss of generality, we assume that $U^{-1} = U$. Denote by $\mathscr{Z}$ the set of all those families $(F_{E,n},D_{E,n},(\phi_{E,n,g})_{g \in E\setminus \{ e \}})_{(E,n) \in J}$ which satisfy the following conditions: \begin{enumerate}
	\item[(i)] $J \subseteq I$ is $<$-downward closed, i.e., $(E',n') \in J$ whenever $(E',n') < (E,n) \in J$, 
	\item[(ii)] $D_{E,n} \subseteq F_{E,n} \in \mathscr{F}(G) \setminus \{ \emptyset \}$ and $\vert D_{E,n} \vert \geq \left(1-\tfrac{1}{n}\right)\vert F_{E,n} \vert$ for all $(E,n) \in J$,
	\item[(iii)] $F_{E,n} \cap gF_{E,n} = \emptyset$ for all $(E,n) \in J$ and $g \in E\setminus \{ e \}$,
	\item[(iv)] $\phi_{E,n, g} \colon D_{E,n} \to gF_{E,n}$ is an injection for all $(E,n) \in J$ and $g \in E\setminus \{ e \}$,
	\item[(v)] $\phi_{E,n,g} (x) \in Ux$ for all $(E,n) \in J$, $g \in E\setminus \{ e \}$, and $x \in D_{E,n}$,
	\item[(vi)] $EF_{E,n} \cap E'F_{E',n'} = \emptyset$ for any two $(E,n), (E',n') \in J$ with $(E,n) \ne (E',n')$.
\end{enumerate} Furthermore, let us define a partial order $\preceq$ on $\mathscr{Z}$ by setting \begin{displaymath}
	(F_{E,n},D_{E,n},(\phi_{E,n,g})_{g \in E\setminus \{ e \}})_{(E,n) \in J} \ \preceq \ (F'_{E,n},D'_{E,n},(\phi'_{E,n,g})_{g \in E\setminus \{ e \}})_{(E,n) \in J'}
\end{displaymath} if $J \subseteq J'$ and $F_{E,n} = F'_{E,n}$, $D_{E,n} = D'_{E,n}$, $\phi_{E,n,g} = \phi'_{E,n,g}$ for all $(E,n) \in J$ and $g \in E$. It is straightforward to check that $(\mathscr{Z},{\preceq})$ is an inductively ordered set. Hence, Zorn's lemma asserts the existence of a maximal element $(F_{E,n},D_{E,n},(\phi_{E,n,g})_{g \in E\setminus \{ e \}})_{(E,n) \in J} \in \mathscr{Z}$.
		
We show that $J = I$. For contradiction, suppose that $I \ne J$. Let $(E_{0},n_{0})$ be the smallest element of $({I\setminus J},{<})$. By clause~(i), we have $J = \{ (E,n) \in I \mid (E,n) < (E_{0},n_{0}) \}$. Due to our assumption about the well-ordering, it follows that $\vert J \vert < \vert I \vert$. By Lemma~\ref{lemma:nice.folner.sets}, there exist finite non-empty subsets $D \subseteq F \subseteq G$ and injections $\phi_{g} \colon D \to gF$ ($g \in E_{0}\setminus \{ e \}$) with \begin{enumerate}
	\item[$\bullet$] $\vert D \vert \geq \left( 1 - \tfrac{1}{n_{0}} \right) \vert F \vert$,
	\item[$\bullet$] $F \cap gF = \emptyset$ for every $g \in E_{0}\setminus \{ e \}$,
	\item[$\bullet$] $\phi_{g}(x) \in Ux$ for all $g \in E_{0}\setminus \{ e \}$ and $x \in D$.
\end{enumerate} Let $B \defeq \bigcup \{ EF_{E,n} \mid (E,n) \in J \}$. Notice that either $J$ is finite and therefore $(E_{0}F)^{-1}B$ is finite as well, or $J$ is infinite and then $\vert (E_{0}F)^{-1}B \vert = \vert J \vert$. Since $G$ is infinite and $\vert J \vert < \vert I \vert$, either way we have $\vert (E_{0}F)^{-1}B \vert < \vert G \vert$. Thus, $(E_{0}F)^{-1}B \ne G$. Consequently, we may choose $z \in G\setminus ((E_{0}F)^{-1}B)$. Clearly, $E_{0}Fz \cap B = \emptyset$. Define $F_{E_{0},n_{0}} \defeq Fz$, $D_{E_{0},n_{0}} \defeq Dz$, and \begin{displaymath}
	\phi_{E_{0},n_{0},g} \colon D_{E_{0},n_{0}} \to gF_{E_{0},n_{0}}, \quad x \mapsto \phi_{g}(xz^{-1})z
\end{displaymath} for $g \in E_{0}\setminus \{ e \}$. Let $J' \defeq J \cup \{ (E_{0},n_{0}) \}$. We claim that $(F_{E,n},D_{E,n},(\phi_{E,n,g})_{g \in E\setminus \{ e \}})_{(E,n) \in J'}$ is a member of $\mathscr{Z}$. By construction, it satisfies the conditions (i)--(iv) and (vi). Moreover, if $(E,n) \in J'$, then either $(E,n) \in J$ and thus $\phi_{E,n,g} (x) \in Ux$ for all $g \in E\setminus \{ e \}$ and $x \in D_{E,n}$, or $(E,n) = (E_{0},n_{0})$ and then \begin{displaymath}
	\phi_{E,n,g}(x) = \phi_{E_{0},n_{0},g}(x) = \phi_{g}(xz^{-1})z \in Uxz^{-1}z = Ux
\end{displaymath} for all $g \in E\setminus \{ e \}$ and $x \in D_{E,n}$. This shows that $(F_{E,n},D_{E,n},(\phi_{E,n,g})_{g \in E\setminus \{ e \}})_{(E,n) \in J'}$ has the properties (i)--(vi) listed above and hence belongs to $\mathscr{Z}$, which clearly contradicts the maximality of $(F_{E,n},D_{E,n},(\phi_{E,n,g})_{g \in E\setminus \{ e \}})_{(E,n) \in J}$. Thus, $I = J$.
		
Finally, let us define a map $\psi \colon G \to \mathrm{Sym}(G)$ by setting $\psi (e) \defeq \id_{G}$ and \begin{displaymath}
	\psi (g)(x) \defeq \begin{cases}
			\phi_{E,n,g}(x) & \text{if } g \in E \text{ and } x \in D_{E,n} \text{ for some } (E,n) \in I , \\
			\phi_{E,n,g}^{-1}(x) & \text{if } g \in E \text{ and } x \in \phi_{E,n,g}(D_{E,n}) \text{ for some } (E,n) \in I , \\
			x & \text{otherwise}
		\end{cases}
\end{displaymath} for $g, x \in G$, $g \ne e$. We note that $\psi$ is well defined due to clauses (ii)--(iv) and (vi) above. Moreover, $\{ \psi(g)(x)x^{-1} \mid g,x \in G \} \subseteq U$ by clause (v) and symmetry of $U$. Hence, \begin{displaymath}
	\alpha \colon G \to \mathrm{Sym}(G), \quad g \mapsto \psi (g) \circ \lambda_{g}
\end{displaymath} is a member of $\mathscr{N}_{G}(U)$. We are left to prove that the action of $\Gamma (\alpha)$ on $G$ is amenable. To this end, we will verify the amenability criterion in Remark~\ref{remark:amenable.actions} with respect to the generating set $S \defeq \{ \alpha (g) \mid g \in G \}$. Let $\theta > 1$ and consider a finite subset $E' \subseteq S$. Take a finite subset $E \subseteq G$ with $e \in E$ and $E' \subseteq \{ \alpha (g) \mid g \in E \}$. Choose $n \in \mathbb{N}\setminus \{ 0 \}$ so that $1 + \frac{\vert E \vert}{n} \leq \theta$. Then \begin{displaymath}
	E'F_{E,n} \subseteq \bigcup \{ \alpha (g)(F_{E,n}) \mid g \in E \} \subseteq F_{E,n} \cup \bigcup \{ (gF_{E,n})\setminus \phi_{E,n,g}(D_{E,n}) \mid g \in E\setminus \{ e \} \} ,
\end{displaymath} and hence \begin{displaymath}
	\vert E'F_{E,n} \vert \leq \vert F_{E,n} \vert + \vert E \vert (\vert F_{E,n} \vert - \vert D_{E,n} \vert ) \leq \left(1+\tfrac{\vert E \vert}{n}\right) \vert F_{E,n} \vert \leq \theta \vert F_{E,n} \vert
\end{displaymath} by clauses (ii) and (iv) above. Thus, the action is amenable by Theorem~\ref{theorem:amenable.actions}.

(2)$\Longrightarrow$(3). Let $U \in \mathscr{U}_{e}(G)$. Take any $\alpha \in \mathscr{N}_{G}(U)$ and consider the associated injection \begin{displaymath}
	f_{\alpha} \colon G \to G \times \mathscr{N}_{G}(U), \quad x \mapsto (x,\alpha) .
\end{displaymath} Evidently, $f_{\alpha}(G) = G \times \{ \alpha \}$ is $\Gamma (\alpha)$-invariant and $h \colon \Gamma_{U}(G) \to \Gamma (\alpha), \, \gamma \mapsto f_{\alpha}^{-1} \circ \gamma \circ f_{\alpha}$ is a well-defined homomorphism satisfying $f_{\alpha} \circ h(\gamma) = \gamma \circ f_{\alpha}$ for every $\gamma \in \Gamma_{U}(G)$. Therefore, if the action of $\Gamma (\alpha)$ on $G$ is amenable, then so is the action of $\Gamma_{U}(G)$ on $f_{\alpha}(G)$ and hence on $G \times \mathscr{N}_{G}(U)$. This proves the desired implication.
		
(3)$\Longrightarrow$(1). We are going to utilize Theorem~\ref{theorem:topological.folner}. To this end, let $\theta \in (0,1)$, $U \in \mathscr{U}_{e}(G)$, and $E \in \mathscr{F}(G)$. By assumption, $\Gamma_{U}(G)$ acts amenably on $G \times \mathscr{N}_{G}(U)$. According to Theorem~\ref{theorem:amenable.actions}, there thus exists a finite non-empty subset $F' \subseteq G \times \mathscr{N}_{G}(U)$ such that $\vert E'F' \vert \leq (2-\theta)\vert F' \vert$ for $E' \defeq \{ (\bigoplus_{\alpha \in \mathscr{N}_{G}(U)} \alpha)(g) \mid g \in E \}$. It follows that there exists some $\alpha \in \mathscr{N}_{G}(U)$ such that $F \defeq f_{\alpha}^{-1}(F')$ is non-empty and $\vert E''F \vert \leq (2-\theta)\vert F \vert$ for $E'' \defeq \{ \alpha (g) \mid g \in E \}$. Let $g \in E$. Consider $D \defeq F \cap \alpha(g)(F)$ and the injection $\phi \colon D \to gF, \, x \mapsto g\alpha(g)^{-1}(x)$. Observe that \begin{displaymath}
	\vert D \vert = \vert F \vert - \vert \alpha(g)(F) \setminus F \vert \geq \vert F \vert - (\vert E''F \vert - \vert F \vert ) \geq \theta \vert F \vert .
\end{displaymath} For all $x \in G$, we have $gx \in U\alpha(g)(x)$, and so $g\alpha(g)^{-1}(x) \in Ux$. Hence, $\phi (x) \in Ux$ for all $x \in D$. Thus, $\mu (F,gF,U) \geq \theta \vert F \vert$. Due to Theorem~\ref{theorem:topological.folner}, this means that $G$ is amenable. \end{proof}
	
\begin{remark} The map $\alpha$ constructed for the proof of the implication (1)$\Longrightarrow$(2) above has the additional property that $\alpha (g) \circ \lambda_{g^{-1}} = \psi (g)$ is an involution for every $g \in G$. \end{remark}
	
Before moving on to our topological solution to von Neumann's problem (Corollary~\ref{corollary:von.neumann}), we want to briefly discuss Theorem~\ref{theorem:approximating.actions} for compact groups. This shall be done by proving the proposition below. Recall that a topological group $G$ is \emph{precompact} if, for every $U \in \mathscr{U}_{e}(G)$, there exists a finite subset $F \subseteq G$ such that $UF = G$.
	
\begin{prop}\label{proposition:precompact} A topological group $G$ is precompact if and only if, for every identity neighborhood $U \in \mathscr{U}_{e}(G)$, there exists $\alpha \in \mathscr{N}_{G}(U)$ such that $\Gamma (\alpha)$ is finite. \end{prop}
	
\begin{proof} ($\Longrightarrow$) Without loss of generality we assume that $G$ is infinite. Let $U \in \mathscr{U}_{e}(G)$. Fix $V \in \mathscr{U}_{e}(G)$ such that $V^{-1} = V$ and $V^{3} \subseteq U$. Since every precompact topological group is SIN, we may also assume that $gV = Vg$ for all $g \in G$. By $G$ being precompact, there exists a finite subset $F \subseteq G$ such that $FF^{-1} \cap V = \{ e \}$ and \begin{displaymath}
	\vert F \vert = \sup \left\{ \vert E \vert \left\vert \, E \subseteq G , \, EE^{-1} \cap V = \{ e \} \right\} . \right.
\end{displaymath} Let $g \in G$. As $gV = Vg$, it follows that $gF(gF)^{-1} \cap V = g(FF^{-1} \cap V)g^{-1} = \{ e \}$. We prove that the bipartite graph $\mathscr{B} \defeq \mathscr{B}(F,gF,V)$ admits a perfect matching. To this end, let $S \subseteq F$ and $T \defeq N_{\mathscr{B}}(S)$. Consider the set $E \defeq ((gF)\setminus T) \cup S$. If $x,y \in E$ and $xy^{-1} \in V$, then either $\{ x,y \} \subseteq gF$ and thus $x=y$ by the above, or $\{ x,y \} \cap S \ne \emptyset$ and hence $\{ x,y \} \subseteq S$, which also implies that $x=y$. Therefore, $EE^{-1} \cap V = \{ e \}$. It follows that \begin{displaymath}
	\vert F \vert \geq \vert E \vert = \vert (gF)\setminus T \vert + \vert S \vert = \vert F \vert - \vert T \vert + \vert S \vert
\end{displaymath} i.e., $\vert S \vert \leq \vert T \vert$. Consequently, $\mathscr{B}$ admits a perfect matching $\phi_{g} \colon F \to gF$ by Corollary~\ref{corollary:hall}. Let $W \in \mathscr{U}_{e}(G)$ such that $W^{-1} = W$ and $W^{2} \subseteq V$. Since $G = VF$ and $Wx \cap Wy = \emptyset$ for any two distinct $x,y \in F$, we may choose a map $\pi \colon G \to F$ such that \begin{enumerate}
	\item[$\bullet$] $Wx \subseteq \pi^{-1}(x)$ for every $x \in F$,
	\item[$\bullet$] $\pi (x) \in Vx$ for all $x \in G$.
\end{enumerate} Since $G$ is infinite and precompact, every non-empty open subset of $G$ has the same cardinality as $G$. This readily implies that $\vert \pi^{-1}(x) \vert = \vert G \vert$ for every $x \in F$. For each $x \in F$, choose a bijection $\beta_{x} \colon \pi^{-1}(x) \to G$. Let us define a map $h \colon \mathrm{Sym}(F) \to \mathrm{Sym}(G)$ by setting \begin{displaymath}
	h(\gamma )(x) \defeq \beta_{\gamma \pi (x)}^{-1}(\beta_{\pi (x)}(x)) \qquad (\gamma \in \mathrm{Sym}(F), \, x \in G) .
\end{displaymath} Note that $\pi (h(\gamma)(x)) = \pi (\beta_{\gamma \pi (x)}^{-1}(\beta_{\pi (x)}(x))) = \gamma \pi (x)$ for all $\gamma \in \mathrm{Sym}(F)$ and $x \in G$. It follows that $h$ is a homomorphism: if $\gamma, \gamma' \in \mathrm{Sym}(F)$ and $x \in G$, then \begin{align*}
	h(\gamma'\gamma)(x) &= \beta^{-1}_{\gamma'\gamma \pi (x)}(\beta_{\pi (x)}(x)) = \beta^{-1}_{\gamma'\gamma \pi (x)}(\beta_{\gamma \pi (x)}(\beta^{-1}_{\gamma \pi (x)}(\beta_{\pi (x)}(x)))) \\
		&= \beta^{-1}_{\gamma'\pi (h(\gamma)(x))}(\beta_{\pi (h(\gamma)(x))}(h(\gamma)(x))) = h(\gamma')(h(\gamma)(x)) .
\end{align*} Now, for each $g \in G$, let us define $\gamma (g) \in \mathrm{Sym}(F)$ by \begin{displaymath}
	\gamma (g)(x) \defeq \phi^{-1}_{g}(gx) \qquad (x \in F) ,
\end{displaymath} and put $\alpha (g) \defeq h(\gamma (g)) \in \mathrm{Sym}(G)$. Then $\{ \alpha (g) \mid g \in G \} = \{ h(\gamma (g)) \mid g \in G\} \subseteq h(\mathrm{Sym}(F))$. Hence, $\Gamma (\alpha)$ is contained in the finite group $h(\mathrm{Sym}(F))$ and therefore finite itself. It remains to prove that $\alpha \in \mathscr{N}_{G}(U)$. Let $g,x \in G$. By choice of $\beta_{\gamma(g)\pi (g^{-1}x)}$, $\phi_{g}$, and $\pi$, it follows that \begin{align*}
	\alpha(g)(x) = h(\gamma(g))(x) = \beta^{-1}_{\gamma(g)\pi(x)} (\beta_{\pi(x)}(x)) \in V\gamma(g)\pi (x) &= V\phi_{g}^{-1}(g\pi (x)) \subseteq V^{2}g\pi (x) \\
		&\subseteq V^{2}gVx = V^{3}gx \subseteq Ugx .
\end{align*} This proves that $\alpha$ is contained in $\mathscr{N}_{G}(U)$, which completes the argument.
		
($\Longleftarrow$) Let $U \in \mathscr{U}_{e}(G)$. By our hypothesis, there is $\alpha \in \mathscr{N}_{G}(U^{-1})$ such that $\Gamma (\alpha)$ is finite. Consider the finite subset $F \defeq \{ \alpha(g)(e) \mid g \in G \} \subseteq G$. Clearly, if $g \in G$, then $\alpha(g)(e) \in U^{-1}g$ and thus $g \in U\alpha(g)(e) \subseteq UF$. That is, $G = UF$. This shows that $G$ is precompact. \end{proof}

Next up we turn our attention towards von Neumann's problem. In his seminal work on amenable groups~\cite{vonNeumann} von Neumann observed that amenability of a group is inherited by each of its subgroups and that the free group $F_{2}$ on two generators is not amenable. The question whether every non-amenable group would contain a subgroup isomorphic to $F_{2}$ was first posed in print by Day~\cite{day57}, however became known as the \emph{von Neumann problem}, and was solved in the negative by Ol'{\v{s}}anski{\u\i}~\cite{olshanskii}. Despite this negative answer to the original question, there are a number of positive solutions to very interesting variations of the von Neumann problem, including the ones by Whyte~\cite{whyte} in terms of geometric group theory, by Gaboriau and Lyons~\cite{GaboriauLyons} in the context of measured group theory, as well as by Marks and Unger~\cite{MarksUnger} concerning measurable dynamics. Whyte~\cite{whyte} proved that a finitely generated group $G$ is non-amenable if and only if $F_{2}$ admits a translation-like action on $G$. Our Corollary~\ref{corollary:von.neumann} generalizes Whyte's result to the realm of topological groups and therefore may be regarded as a topological solution to von Neumann's problem.

We are going to state Whyte's result (Theorem~\ref{theorem:whyte}) as well as our topological variant (Corollary~\ref{corollary:von.neumann}) in terms of wobbling groups.

\begin{definition} The \emph{wobbling group} of a group action $(G,X)$ is defined to be \begin{displaymath}
	W(G,X) \defeq \{ \gamma \in \mathrm{Sym}(X) \mid \exists E \in \mathscr{F}(G)\colon \, \gr (\gamma) \subseteq \{ (x,gx) \mid x \in X, \, g \in E \} \} .
\end{displaymath} Given sets $X$ and $I$ and a map $\alpha \colon I \to \mathrm{Sym}(X)$, we will abbreviate $W(\alpha) \defeq W(\Gamma (\alpha),X)$. Considering the action of an arbitrary group $G$ on itself by left translations, we define the \emph{wobbling group} of $G$ to be $W(G) \defeq W(G,G)$. If $G$ is a topological group and $U \in \mathscr{U}_{e}(G)$, then we define the \emph{$U$-wobbling group} of $G$ to be \begin{displaymath}
	W_{U}(G) \, \defeq \, W(\Gamma_{U}(G),G\times \mathscr{N}_{G}(U)) \, = \, W \! \left( \bigoplus\nolimits_{\alpha \in \mathscr{N}_{G}(U)} \alpha \right) .
\end{displaymath} \end{definition}

Whyte's geometric solution~\cite{whyte} to the von Neumann problem can be stated as follows, where a subgroup $G$ of the full symmetric group $\mathrm{Sym}(X)$ of a set $X$ is said to be \emph{semi-regular} if no non-identity element of $G$ has a fixed point in $X$.

\begin{thm}[Theorem 6.1 in~\cite{whyte}]\label{theorem:whyte} A group $G$ is non-amenable if and only if the free group $F_{2}$ on two generators is isomorphic to a semi-regular subgroup of $W(G)$. \end{thm}

Of course, if a topological group $G$ is not amenable, then it must not be amenable as a discrete group either, and hence $F_{2}$ embeds into $W(G)$ as a semi-regular subgroup by Theorem~\ref{theorem:whyte}. However, topological non-amenability is a stronger assertion than discrete non-amenability, and therefore one might ask for a strengthening of Whyte's result that characterizes non-amenability for topological groups. We resolve this question as follows.

\begin{cor}\label{corollary:von.neumann} Let $G$ be a Hausdorff topological group. The following are equivalent. \begin{enumerate}
	\item[$(1)$] $G$ is not amenable.
	\item[$(2)$] There exists $U \in \mathscr{U}_{e}(G)$ such that, for every $\alpha \in \mathscr{N}_{G}(U)$, the free group $F_{2}$ is isomorphic to a semi-regular subgroup of $W(\alpha)$.
	\item[$(3)$] $F_{2}$ is isomorphic to a semi-regular subgroup of $W_{U}(G)$ for some $U \in \mathscr{U}_{e}(G)$.
\end{enumerate} \end{cor}

\begin{proof} It is not difficult to see that Whyte's proof of Theorem~\ref{theorem:whyte} works just as well for group actions: a group action $(G,X)$ is non-amenable if and only if the free group $F_{2}$ is isomorphic to a semi-regular subgroup of $W(G,X)$. This observation combined with our Theorem~\ref{theorem:approximating.actions} immediately implies the desired corollary. \end{proof}

\begin{remark}\label{remark:von.neumann} Let $G$ be a non-amenable Hausdorff topological group. As due to~(3) in Corollary~\ref{corollary:von.neumann}, there exists $U \in \mathscr{U}_{e}(G)$ such that $F_{2}$ is isomorphic to a semi-regular subgroup of $W_{U}(G)$, i.e., there exist $g_{1},\ldots ,g_{m},h_{1},\ldots ,h_{n} \in F_{G}$ such that, for every $\alpha \in \mathscr{N}_{G}(U)$, there exist partitions $G = \bigcupdot_{i=1}^{m} A_{i} = \bigcupdot_{j=1}^{n} B_{j}$ so that the maps \begin{displaymath}
	\coprod_{i=1}^{m} \alpha^{\ast}(g_{i})\vert_{A_{i}} \colon G \to G , \qquad \coprod_{j=1}^{n} \alpha^{\ast}(h_{j})\vert_{B_{j}} \colon G \to G
\end{displaymath} are elements of $\mathrm{Sym} (G)$ generating a semi-regular subgroup being isomorphic to $F_{2}$. That is to say, similarly to the situation with Theorem~\ref{theorem:approximating.actions} explained in Remark~\ref{remark:paradoxical}, statement~(3) of Corollary~\ref{corollary:von.neumann} constitutes a uniform version of statement~(2). \end{remark}

It is well known that containment of a discrete subgroup being isomorphic to $F_{2}$ does not prevent a general (i.e., not necessarily locally compact) topological group from being amenable: among the most prominent examples of amenable topological groups admitting discrete free subgroups are the full symmetric group of any infinite set with the topology of point-wise convergence, the unitary group of any infinite-dimensional Hilbert space equipped with the strong operator topology (which is even extremely amenable by a famous result of Gromov and Milman~\cite{GromovMilman}), as well as the automorphism group $\Aut (\mathbb{Q},{<})$ with the topology of point-wise convergence (which Pestov proved both to be extremely amenable and to contain a discrete copy of $F_{2}$~\cite{pestov}). Recently, the first example of an extremely amenable Polish group admitting a complete bi-invariant metric and containing an (even maximally) discrete free subgroup was given by Carderi and the second author~\cite{CarderiThom}.

In particular, amenability of a topological group is not necessarily inherited by its closed subgroups. Of course, this is very different to the discrete case where amenability is passed on to subgroups. However, Corollary~\ref{corollary:von.neumann} suggests another perspective on this inheritance problem, which we record in Corollary~\ref{corollary:strong.subgroups}. The proof of Corollary~\ref{corollary:strong.subgroups} is a straightforward combination of Theorem~\ref{theorem:approximating.actions} and the following well-known fact.

\begin{lem}\label{lemma:wobbling} Consider a group $G$ acting on a non-empty set $X$. A mean $\mu \colon \ell^{\infty}(X) \to \mathbb{R}$ is $G$-invariant if and only if $\mu$ is $W(G,X)$-invariant. \end{lem}

\begin{proof} Clearly, $W(G,X)$-invariance implies $G$-invariance. To prove the converse, suppose that $\mu$ is $G$-invariant. Let $h \in W(G,X)$. Then there is a finite partition $\mathscr{P}$ of $G$ along with $(g_{P})_{P \in \mathscr{P}} \in G^{\mathscr{P}}$ such that $h\vert_{P} = g_{P}\vert_{P}$ for each $P \in \mathscr{P}$. For $f \in \ell^{\infty}(X)$, it follows that \begin{align*}
	\mu (f \circ h) \, &= \sum_{P \in \mathscr{P}} \mu \left( (f \circ h) \cdot \mathbf{1}_{P} \right) = \sum_{P \in \mathscr{P}} \mu \left( (f \circ g_{P}) \cdot \mathbf{1}_{P} \right) = \sum_{P \in \mathscr{P}} \mu \left( f \cdot \left(\mathbf{1}_{P} \circ g_{P}^{-1} \right)\right) \\
	&= \sum_{P \in \mathscr{P}} \mu \left( f \cdot (\mathbf{1}_{g_{P} P})\right) = \sum_{P \in \mathscr{P}} \mu \left( f \cdot \mathbf{1}_{\alpha(P)} \right) = \mu (f) .\qedhere
\end{align*} \end{proof}

Since any subgroup of a group $G$ embeds into $W(G)$ as a semi-regular subgroup, the hypothesis of the following corollary is satisfied for discrete (amenable) groups.

\begin{cor}\label{corollary:strong.subgroups} Let $G$ be an amenable Hausdorff topological group and let $H$ be any group. Suppose that there exists $U \in \mathscr{U}_{e}(G)$ such that, for every $\alpha \in \mathscr{N}_{G}(U)$, the group $H$ is isomorphic to a semi-regular subgroup of $W(\alpha)$ (which is true, e.g., if $H$ is isomorphic to a semi-regular subgroup of $W_{U}(G)$). Then $H$ is amenable. \end{cor}

\begin{proof} By Theorem~\ref{theorem:approximating.actions}, there exists $\alpha \in \mathscr{N}_{G}(U)$ such that the action of $\Gamma (\alpha)$ on $G$ is amenable. Let $\mu \colon \ell^{\infty}(G) \to \mathbb{R}$ be an $\Gamma (\alpha)$-invariant mean. By Lemma~\ref{lemma:wobbling}, it follows that $\mu$ is $W(\alpha)$-invariant. By assumption, there is an embedding $\phi \colon H \to W(\alpha)$ such that $\phi (H)$ is semi-regular. Hence, there exists a map $\psi \colon G \to H$ such that $\psi (\phi (h)(x)) = h\psi (x)$ for all $h \in H$ and $x \in G$. Now, define $\nu \colon \ell^{\infty}(H) \to \mathbb{R}$ by \begin{displaymath}
	\nu (f) \defeq \mu (f \circ \psi)  \qquad (f \in \ell^{\infty}(H)) .
\end{displaymath} It is easy to see that $\nu$ is a mean. Furthermore, the $W(\alpha)$-invariance of $\mu$ implies that \begin{displaymath}
	\nu (f \circ \lambda_{h}) = \mu (f \circ \lambda_{h} \circ \psi) = \mu (f \circ \psi \circ \phi (h)) = \mu (f \circ \psi) = \nu (f)
\end{displaymath} for all $f \in \ell^{\infty}(H)$ and $h \in H$. Thus, $\nu$ is $H$-invariant, as desired. \end{proof}

The remainder of this section will be devoted to refining the previous constructions for Polish groups to incorporate the results of Marks and Unger~\cite{MarksUnger} and in turn provide measurable versions of Theorem~\ref{theorem:approximating.actions}, Remark~\ref{remark:paradoxical}, and Corollary~\ref{corollary:von.neumann}. Evidently, if a topological group $G$ is a metrizable, then so is the topology of uniform convergence on $\mathrm{Sym}(G)^{G}$, and therefore it suffices to consider approximations of the left translation action by sequences. Moreover, in case $G$ is separable, the following observation allows us to concentrate on Borel measurable approximations.

\begin{prop}\label{proposition:separable.groups} Let $G$ be an amenable, separable Hausdorff topological group $G$, let $H$ be a dense countable subgroup of $G$, and let $U \in \mathscr{U}_{e}(G)$. There is $\alpha \in \mathscr{N}_{G}(U)$ such that \begin{enumerate}
	\item[$\bullet$] $\alpha$ has countable image and is Borel measurable,
	\item[$\bullet$] $H$ is $\Gamma (\alpha)$-invariant and $\Gamma (\alpha)$ acts amenably on the set $H$,
	\item[$\bullet$] $\Gamma (\alpha)$ acts on $G\setminus H$ by left $H$-translations.
\end{enumerate} In particular, $\Gamma (\alpha)$ acts on $G$ by Borel automorphisms, and $G \times G \to G , \, (g,h) \mapsto \alpha (g)(h)$ is Borel measurable. \end{prop}

\begin{proof} Fix an open identity neighborhood $V \in \mathscr{U}_{e}(G)$ such that $V^{2} \subseteq U$. By density, the amenability of $G$ is equivalent to the amenability of $H$ (carrying the respective subspace topology). Hence, Theorem~\ref{theorem:approximating.actions} implies the existence of a map $\beta \colon H \to \mathrm{Sym}(H)$ such that \begin{enumerate}
	\item[$\bullet$] $\Gamma (\beta)$ acts amenably on the set $H$,
	\item[$\bullet$] $\beta(g)(h) \in Vgh$ for all $g,h \in H$.
\end{enumerate} Extend $\beta$ to a map $\beta' \colon H \to \mathrm{Sym}(G)$ by setting \begin{displaymath}
	\beta' (h)(x) \defeq \begin{cases}
		\beta (h)(x) & \text{if } x \in H , \\
		hx & \text{otherwise}
	\end{cases} \qquad (h \in H, \, x \in G) .
\end{displaymath} Evidently, $H$ is $\Gamma (\beta')$-invariant, and $\Gamma (\beta')$ acts amenably on $H$ and acts by left $H$-translations on $G\setminus H$. As $H$ is countable, $\beta' (h)$ is an automorphism of the Borel space $G$ for each $h \in H$, and it moreover follows that $H \times G \to G , \, (h,x) \mapsto \beta'(h)(x)$ is Borel measurable. Due to $H$ being countable and dense in $G$, a straightforward recursion also provides us with a Borel measurable map $\pi \colon G \to H$ such that $\pi (x) \in Vx$ for all $x \in G$. Hence, the resulting maps $\alpha \defeq \beta' \circ \pi \colon G \to \mathrm{Sym}(G)$ and $G \times G \to G, \, (g,h) \mapsto \alpha(g)(h)$ are Borel measurable, too. Clearly, we are only left to note that $\alpha$ is a member of $\mathscr{N}_{G}(U)$: indeed, if $g,h \in G$, then \begin{displaymath}
	\alpha(g)(h) = \beta'(\pi (g))(h) \in V\pi (g)h \subseteq V^{2} gh \subseteq Ugh . \qedhere
\end{displaymath} \end{proof}

In view of Proposition~\ref{proposition:separable.groups}, we introduce another bit of notation: for a Polish group $G$ and $U \in \mathscr{U}_{e}(G)$, we denote by $\mathscr{M}_{G}(U)$ the set of those $\alpha \in \mathscr{N}_{G}(U)$ where \begin{enumerate}
	\item[$\bullet$] $\alpha$ has countable image and is Borel measurable, and
	\item[$\bullet$] for each $g \in G$, the map $\alpha (g)$ is a Borel automorphism of $G$.
\end{enumerate} Building on work of Marks and Unger~\cite{MarksUnger}, we obtain the subsequent two corollaries.

\begin{cor}\label{corollary:baire.paradoxical} Let $G$ be a Polish group. The following are equivalent. \begin{enumerate}
	\item[$(1)$] $G$ is not amenable.
	\item[$(2)$] There exists $U \in \mathscr{U}_{e}(G)$ such that, for every $\alpha \in \mathscr{M}_{G}(U)$, the Polish space $G$ admits a $\Gamma (\alpha)$-paradoxical decomposition where each piece has the Baire property.
	\item[$(3)$] There exists $U \in \mathscr{U}_{e}(G)$ such that, for every $(\alpha_{n})_{n \in \mathbb{N}} \in \mathscr{M}_{G}(U)^{\mathbb{N}}$, the Polish space $G \times \mathbb{N}$ admits a $\Gamma (\bigoplus_{n \in \mathbb{N}}\alpha_{n})$-paradoxical decomposition where each piece has the Baire property.
\end{enumerate} \end{cor}

\begin{proof} This follows from Theorem~\ref{theorem:approximating.actions}, Proposition~\ref{proposition:separable.groups}, and~\cite[Theorem~1.1]{MarksUnger}. \end{proof}
	
Recall that a map $f \colon X \to Y$ between topological spaces $X$ and $Y$ is \emph{Baire measurable} if the preimage $f^{-1}(B)$ has the Baire property for every Borel set $B \subseteq Y$.
	
\begin{cor}\label{corollary:baire.free} Let $G$ be a Polish group. The following are equivalent. \begin{enumerate}
	\item[$(1)$] $G$ is not amenable.
	\item[$(2)$] There is $U \in \mathscr{U}_{e}(G)$ such that, for every $\alpha \in \mathscr{M}_{G}(U)$, the free group $F_{2}$ is isomorphic to a semi-regular subgroup of Baire measurable elements of $W(\alpha)$.
	\item[$(3)$] There is $U \in \mathscr{U}_{e}(G)$ such that, for every $(\alpha_{n})_{n \in \mathbb{N}} \in \mathscr{M}_{G}(U)^{\mathbb{N}}$, the free group $F_{2}$ is isomorphic to a semi-regular subgroup of Baire measurable elements of $W(\bigoplus_{n \in \mathbb{N}} \alpha_{n})$.
\end{enumerate} \end{cor}

\begin{proof} This follows from Theorem~\ref{theorem:approximating.actions}, Proposition~\ref{proposition:separable.groups}, and~\cite[Theorem~1.2]{MarksUnger}. \end{proof}

\section{Equivariant geometry of amenable topological groups}\label{section:equivariant.geometry}

In this section we give an application of Theorem~\ref{theorem:topological.folner} concerning coarse geometry of topological groups as studied extensively in~\cite{rosendal1,rosendal2}. In fact, we answer a question posed by Rosendal~\cite[Problem~40]{rosendal2} in the affirmative.

For a start, let us recall some terminology from~\cite{rosendal2}. Let $X$ and $Y$ be pseudo-metric spaces and let $\sigma \colon X \to Y$. The \emph{expansion modulus} of $\sigma$ is defined as $\theta_{\sigma} \colon [0,\infty] \to [0,\infty]$ with \begin{displaymath}
\theta_{\sigma}(t) \defeq \sup \{ d_{Y}(\sigma (x),\sigma (y)) \mid x,y \in X, \, d_{X}(x,y) \leq t \} \qquad (t \geq 0) .
\end{displaymath} Note that $\sigma \colon X \to Y$ is uniformly continuous if and only if $\lim_{t \to 0} \theta_{\sigma}(t) = 0$. We will say that $f \colon X \to Y$ is \emph{bornologous} if $\theta_{\sigma}(t) < \infty$ for all $t < \infty$. Furthermore, the \emph{compression modulus} of $\sigma$ is defined to be $\kappa_{\sigma} \colon [0,\infty] \to [0,\infty]$ with \begin{displaymath}
\kappa_{\sigma}(t) \defeq \inf \{ d_{Y}(\sigma (x),\sigma (y)) \mid x,y \in X, \, d_{X}(x,y) \geq t \} \qquad (t \geq 0) .
\end{displaymath} Moreover, we define the \emph{exact compression modulus} of $\sigma$ as $\tilde{\kappa}_{\sigma} \colon [0,\infty] \to [0,\infty]$ with \begin{displaymath}
\tilde{\kappa}_{\sigma}(t) \defeq \inf \{ d_{Y}(\sigma (x),\sigma (y)) \mid x,y \in X, \, d_{X}(x,y) = t \} \qquad (t \geq 0) .
\end{displaymath} Evidently, $\kappa_{\sigma}(t) = \inf_{s \geq t} \tilde{\kappa}_{\sigma}(s)$ for every $t \geq 0$, and \begin{displaymath}
\tilde{\kappa}_{\sigma}(d_{X}(x,y)) \leq d_{Y}(\sigma (x),\sigma (y)) \leq \theta_{\sigma}(d_{X}(x,y))	
\end{displaymath} for all $x,y \in X$. For a detailed discussion of the functions introduced in this paragraph the reader is referred to~\cite{rosendal2}.

We will also need some additional terminology concerning Banach spaces. As in~\cite{rosendal2}, we say that a Banach space $X$ is \emph{finitely representable} in a Banach space $Y$ if, for every finite-dimensional linear subspace $F \subseteq X$ and every $\epsilon > 0$, there is a linear embedding $T \colon F \to Y$ such that $\Vert T \Vert , \Vert T^{-1} \Vert < 1+\epsilon$. It is quite well known that a Banach space $X$ is finitely representable in a Banach space $Y$ if and only if $X$ embeds isometrically into some ultrapower of $Y$ (see for instance~\cite{stern1,stern2}).

What is more, if $E$ is a Banach space and $1 \leq p < \infty$, then we denote by $L^{p}(E)$ the Banach space of equivalence classes of Bochner measurable functions $f \colon [0,1] \to E$ with \begin{displaymath}
\Vert f \Vert_{L^{p}(E)} = \left( \int_{0}^{1} \Vert f(t) \Vert_{E}^{p} \, \mathrm{d}t \right)^{1/p} < \infty .
\end{displaymath}

For the last bit of terminology, let $\pi$ be a continuous isometric linear representation of a topological group $G$ on a Banach space $X$, that is, $\pi$ is a continuous homomorphism from $G$ into the topological group $\Iso (X)$ of all linear isometries of $X$ endowed with the strong operator topology. As usual, by a \emph{cocycle} for $\pi$ we mean a map $b\colon G \to X$ satisfying $b(xy) = \pi(x)b(y) + b(x)$ for all $x,y \in G$.

Our next theorem constitutes a generalization of Theorem~39 of~\cite{rosendal2} and provides a positive solution to Problem~40 of~\cite{rosendal2}. In turn, Theorem~\ref{theorem:rosendal} generalizes earlier results by Pestov for locally finite groups~\cite[Theorem~4.2]{Pestov08} and by Naor and Peres for finitely generated amenable groups~\cite[Theorem~9.1]{NaorPeres}.

\begin{thm}\label{theorem:rosendal} Let $G$ be an amenable Hausdorff topological group, $E$ be a Banach space. Let $1 \leq p < \infty$. There exists a continuous isometric linear representation $\pi$ of $G$ on a Banach space $X$, finitely representable in $L^{p}(E)$, such that the following holds: if $d$ is a continuous left-invariant pseudo-metric on $G$ and $\sigma \colon (G,d) \to E$ is a uniformly continuous and bornologous map with exact compression modulus $\tilde{\kappa} = \tilde{\kappa}_{\sigma}$ and expansion modulus $\theta = \theta_{\sigma}$, then $\pi$ admits a cocycle $b \colon G \to X$ such that \begin{displaymath}
	\forall x,y \in G \colon \quad \tilde{\kappa} (d(x,y)) \leq \Vert b(x) - b(y) \Vert_{X} \leq \theta (d(x,y)) .
\end{displaymath} \end{thm}

The proof of Theorem~\ref{theorem:rosendal} given below follows very closely the lines of Rosendal, more precisely the proof of Theorem~39 in~\cite{rosendal2}. Our only modification being the usage of Theorem~\ref{theorem:topological.folner}. However, we decided to include the full proof for the reader's convenience. What is more, we are going to revisit the argument subsequently for extremely amenable groups (Proposition~\ref{proposition:rosendal}).

\begin{remark}\label{remark:inversion} Let $G$ be a topological group. For the proof of Theorem~\ref{theorem:rosendal}, we are going to equip~$G$ with its \emph{left uniformity}, i.e., \begin{displaymath}
	\mathscr{E}_{\ell}(G) \defeq \{ E \subseteq G \times G \mid \exists U \in \mathscr{U}_{e}(G) \, \forall x,y \in G \colon \, x^{-1}y \in U \Longrightarrow (x,y) \in E \} ,
	\end{displaymath} and we will denote by $\mathrm{LUC}_{b}(G)$ the set of all bounded left-uniformly continuous real-valued functions on $G$. Of course, the topology induced by the left uniformity is again just the original topology of $G$. Moreover, $\iota \colon (G,\mathscr{E}_{r}(G)) \to (G,\mathscr{E}_{\ell}(G)), \ x \mapsto x^{-1}$ is an isomorphism of uniform spaces and ${\iota} \circ {\lambda_{g}} = {\rho_{g}} \circ {\iota}$ for all $g \in G$. Hence, any statement about the right uniformity and right-invariant metrics on~$G$ can be translated into an equivalent statement about the left uniformity and left-invariant metrics in straight-forward, translation-compatible manner. \end{remark}

\begin{proof}[Proof of Theorem~\ref{theorem:rosendal}] Since $G$ is amenable, Theorem~\ref{theorem:topological.folner} (along with Remarks~\ref{remark:topological.folner} and~\ref{remark:inversion}) asserts the existence of a family $(F_{i})_{i \in I}$ of finite non-empty subsets of $G$ and a non-principal ultrafilter $\mathscr{U}$ on $I$ with $\lim_{i \to \mathscr{U}} \delta_{F_{i}}(f-f \circ \rho_{g}) = 0$ for all $f \in \mathrm{LUC}_{b}(G)$ and $g \in G$. Let $1 \leq p < \infty$. For each $i \in I$, denote by $\ell^{p}(F_{i},E)$ the Banach space obtained by endowing the vector space $E^{F_{i}}$ with the norm \begin{displaymath}
	\Vert f \Vert_{\ell^{p}(F_{i},E)} \defeq \left( \sum_{x \in F_{i}} \Vert f(x) \Vert^{p}_{E} \right)^{1/p} \qquad \left(f \in E^{F_{i}}\right) .
\end{displaymath} Let $\prod_{i \to \mathscr{U}} \ell^{p}(F_{i},E)$ denote the corresponding ultraproduct. More precisely, equipping \begin{displaymath}
	\left. W \defeq \left\{ (f_{i})_{i \in I} \in \prod_{i \in I} \ell^{p}(F_{i},E) \ \right| \, \sup_{i \in I} \Vert f_{i} \Vert_{\ell^{p}(F_{i},E)} < \infty \right\}
\end{displaymath} with the semi-norm given by \begin{displaymath}
	\Vert (f_{i})_{i \in I} \Vert_{\mathscr{U}} \defeq \lim_{i \to \mathscr{U}} \Vert f_{i} \Vert_{\ell^{p}(F_{i},E)} \qquad ((f_{i})_{i \in I} \in W) 
\end{displaymath} and setting $N \defeq \{ (f_{i})_{i \in I} \in W \mid \Vert (f_{i})_{i \in I} \Vert_{\mathscr{U}} = 0 \}$, we define $\prod_{i \to \mathscr{U}} \ell^{p}(F_{i},E)$ to be the quotient space $W/N$. Let $\mathrm{LUC}_{b}(G,E)$ be the vector space of all bounded left-uniformly continuous maps from $G$ to $E$. Consider the linear operator $\Theta \colon \mathrm{LUC}_{b}(G,E) \to \prod_{i \to \mathscr{U}} \ell^{p}(F_{i},E)$ given by \begin{displaymath}
	\Theta (f) \defeq \left( \vert F_{i} \vert^{-1/p} \cdot f\vert_{F_{i}} \right)_{i \in I} + N \qquad (f \in \mathrm{LUC}_{b}(G,E)) .
\end{displaymath} Now, $\Vert f \Vert_{\mathscr{U},p} \defeq \Vert \Theta (f) \Vert_{\mathscr{U}}$ $(f \in \mathrm{LUC}_{b}(G,E))$ defines a semi-norm on $\mathrm{LUC}_{b}(G,E)$. Moreover, \begin{align*}
	\Vert f \Vert^{p}_{\mathscr{U},p} &= \left\Vert \left( \vert F_{i} \vert^{-1/p} \cdot f\vert_{F_{i}} \right)_{i \in I} \right\Vert^{p}_{\mathscr{U}} = \left( \lim_{i \to \mathscr{U}} \left\Vert \vert F_{i} \vert^{-1/p} \cdot f\vert_{F_{i}} \right\Vert_{\ell^{p}(F_{i},E)} \right)^{p} \\
	&= \lim_{i \to \mathscr{U}} \sum_{x \in F_{i}} \left\Vert \vert F_{i} \vert^{-1/p} f(x) \right\Vert_{E}^{p} = \lim_{i \to \mathscr{U}} \frac{1}{\vert F_{i} \vert} \sum_{x \in F_{i}} \Vert f(x) \Vert_{E}^{p} 
\end{align*} for every $f \in \mathrm{LUC}_{b}(G,E)$. We claim that $\Vert \cdot \Vert_{\mathscr{U},p}$ is invariant under the linear representation $\pi$ of $G$ on $\mathrm{LUC}_{b}(G,E)$ given by \begin{displaymath}
	\pi (g)f \defeq f \circ \rho_{g^{-1}} \qquad (g \in G, \, f \in \mathrm{LUC}_{b}(G,E)) .
\end{displaymath} Indeed, if $f \in \mathrm{LUC}_{b}(G,E)$, then the function $f^{\ast} \colon G \to \mathbb{R}, \, x \mapsto \Vert f(x) \Vert_{E}^{p}$ is bounded and left-uniformly continuous, and hence \begin{align*}
	\Vert f \Vert_{\mathscr{U},p}^{p} - \Vert \pi (g)f \Vert_{\mathscr{U},p}^{p} &= \lim_{i \to \mathscr{U}} \frac{1}{\vert F_{i} \vert} \sum_{x \in F_{i}} \Vert f(x) \Vert_{E}^{p} - \lim_{i \to \mathscr{U}} \frac{1}{\vert F_{i} \vert} \sum_{x \in F_{i}} \Vert f(xg) \Vert_{E}^{p} \\
	&= \lim_{i \to \mathscr{U}} \delta_{F_{i}}(f^{\ast} - f^{\ast} \circ \rho_{g^{-1}}) = 0
\end{align*} for all $g \in G$. Consider the subspace $M \defeq \{ f \in \mathrm{LUC}_{b}(G,E) \mid \Vert f \Vert_{\mathscr{U},p} = 0 \}$ and denote by~$X$ the completion of $\mathrm{LUC}_{b}(G,E)/M$ with respect to $\Vert \cdot \Vert_{\mathscr{U},p}$. Since $\pi$ is an isometric linear representation of $G$ on $\mathrm{LUC}_{b}(G,E)$, it induces an isometric linear representation of $G$ on $X$, which we continue denoting by $\pi$.
	
We are going to prove that $\pi \colon G \to \Iso (X)$ is continuous with respect to the strong operator topology on $\Iso (X)$, which means that $G \to X, \, g \mapsto \pi (g)x$ is continuous for every $x \in X$. Since $\pi$ is an isometric representation and the quotient $\mathrm{LUC}_{b}(G,E)/M$ is dense in $X$, it suffices to show that $G \to \mathrm{LUC}_{b}(G,E), \, g \mapsto \pi (g)f$ is continuous for every $f \in \mathrm{LUC}_{b}(G,E)$. To see this, let $f \in \mathrm{LUC}_{b}(G)$ and $\epsilon > 0$. Due to $f$ being left-uniformly continuous, there exists $U \in \mathscr{U}_{e}(G)$ such that $\Vert f(g) - f(h) \Vert_{E} \leq \epsilon$ for all $g,h \in G$ with $g^{-1}h \in U$. Hence, \begin{displaymath}
	\Vert \pi (g)f - \pi (h)f \Vert_{\mathscr{U},p}^{p} = \lim_{i \to \mathscr{U}} \frac{1}{\vert F_{i} \vert} \sum_{x \in F_{i}} \Vert f(xg)-f(xh) \Vert_{E}^{p} \leq \epsilon^{p}
\end{displaymath} and therefore $\Vert \pi (g)f - \pi (h)f \Vert_{\mathscr{U},p} \leq \epsilon$ for all $g,h \in G$ with $g^{-1}h \in U$. This proves that $\pi$ is continuous with respect to the strong operator topology.

We argue that $X$ is finitely representable in $L^{p}(E)$. Since $\ell^{p}(F_{i},E)$ is isometrically isomorphic to a linear subspace of $L^{p}(E)$ for each $i \in I$, it follows that $\prod_{i \to \mathscr{U}} \ell^{p}(F_{i},E)$ is finitely representable in $L^{p}(E)$. By construction of norm, $X$ embeds isometrically into $\prod_{i \to \mathscr{U}} \ell^{p}(F_{i},E)$ and is therefore finitely representable in~$L^{p}(E)$.
	
We now claim that the representation $\pi$ has the additional property stated in the theorem. To prove this, let $d$, $\sigma$, $\tilde{\kappa}$, $\theta$ be as in the theorem. Define $b \colon G \to \mathrm{LUC}_{b}(G,E)$ by \begin{displaymath}
	b(x) \defeq \pi (x) \sigma - \sigma \qquad (x \in G) .
\end{displaymath} Note that $b$ is well defined: if $x \in G$, then clearly $b(x)$ is left-uniformly continuous as $\sigma$ is, and $b(x)$ is bounded as \begin{displaymath}
	\sup_{y \in G} \Vert b(x)(y) \Vert_{E} = \sup_{y \in G} \Vert \sigma (yx) - \sigma(y) \Vert_{E} \leq \sup_{y \in G} \theta (d(yx,y)) = \theta (d(x,e)) < \infty
\end{displaymath} due to $d$ being left-invariant and $\sigma$ being bornologous. Again, in terms of notation we will not distinguish between $b$ and the associated map from $G$ to $X$. It is easy to see that $b$ is a cocycle for $\pi$, that is, $b(xy) = \pi (x)b(y) + b(x)$ for all $x,y \in G$. For all $x,y,z \in G$, the left-invariance of $d$ yields that \begin{align*}
	\tilde{\kappa} (d(x,y)) &= \tilde{\kappa} (d(zx,zy)) \leq \Vert \sigma (zx) - \sigma (zy) \Vert_{E} \leq \theta (d(zx,zy)) = \theta (d(x,y)) ,
\end{align*} and since $\sigma (zx) - \sigma (zy) = (\pi(x)\sigma)(z) - (\pi(y)\sigma)(z)$, we arrive at \begin{displaymath}
	\tilde{\kappa} (d(x,y)) \leq \Vert (\pi(x)\sigma)(z) - (\pi(y)\sigma)(z) \Vert_{E} \leq \theta (d(x,y)) .
\end{displaymath} We conclude that \begin{displaymath}
	\tilde{\kappa} (d(x,y))^{p} \leq \frac{1}{\vert F_{i} \vert} \sum_{z \in F_{i}} \Vert (\pi(x)\sigma)(z) - (\pi(y)\sigma)(z) \Vert_{E}^{p} \leq \theta (d(x,y))^{p}
\end{displaymath} for all $x,y \in G$ and $i \in I$. By the expression for $\Vert \cdot \Vert^{p}_{\mathscr{U},p}$, it follows that \begin{displaymath}
	\tilde{\kappa} (d(x,y)) \leq \Vert \pi(x)\sigma - \pi(y)\sigma \Vert_{\mathscr{U},p} \leq \theta (d(x,y)) ,
\end{displaymath} and since $b(x) - b(y) = \pi(x)\sigma - \pi(y)\sigma$, this just means that \begin{displaymath}
	\tilde{\kappa} (d(x,y)) \leq \Vert b(x) - b(y) \Vert_{\mathscr{U},p} \leq \theta (d(x,y)) 
\end{displaymath} for all $x,y \in G$. This completes the proof. \end{proof}

The proof of Theorem~\ref{theorem:rosendal} given above even allows for slight improvement concerning extremely amenable groups, as our next result reveals.

\begin{prop}\label{proposition:rosendal} Let $G$ be an extremely amenable topological group and let $E$ be a Banach space. There exists a continuous isometric linear representation $\pi$ of $G$ on a Banach space~$X$, finitely representable in $E$, such that the following holds: if $d$ is a continuous left-invariant pseudo-metric on $G$ and $\sigma \colon (G,d) \to E$ is a uniformly continuous and bornologous map with exact compression modulus $\tilde{\kappa} = \tilde{\kappa}_{\sigma}$ and expansion modulus $\theta = \theta_{\sigma}$, then $\pi$ admits a cocycle $b \colon G \to X$ such that \begin{displaymath}
	\forall x,y \in G \colon \quad \tilde{\kappa} (d(x,y)) \leq \Vert b(x) - b(y) \Vert_{X} \leq \theta (d(x,y)) .
	\end{displaymath} \end{prop}

\begin{proof} It is well known that a topological group $G$ is extremely amenable if and only if the continuous action of $G$ on its Samuel compactification admits a fixed point (for more details see~\cite{pestov}). Since $G$ is extremely amenable, there thus exists a family $(x_{i})_{i \in I} \in G^{I}$ along with an ultrafilter $\mathscr{U}$ on $I$ such that $\lim_{i \to \mathscr{U}} f(x_{i}) - f(x_{i}g) = 0$ for all $f \in \mathrm{LUC}_{b}(G)$ and $g \in G$. Reviewing the proof of Theorem~\ref{theorem:rosendal} for $F_{i} \defeq \{ x_{i} \}$ $(i \in I)$, we now observe that $\ell^{1}(F_{i},E) \cong E$ for every $i \in I$, whence the constructed Banach space $X$ embeds isometrically into the ultrapower of $E$ with respect to $\mathscr{U}$. That is, $X$ is finitely representable in $E$. \end{proof}

As outlined in~\cite{rosendal2}, Theorem~\ref{theorem:rosendal} has a number of interesting applications concerning the coarse geometry of topological groups. In particular, it implies that Theorem~7, Corollary~41, Corollary~42, and Corollary~44 of~\cite{rosendal2} are still valid if the assumption of F\o lner amenability is replaced by amenability.

\begin{definition}[\cite{rosendal2}]\label{definition:rosendal} A Polish group~$G$ is called \emph{F\o lner amenable} if \begin{enumerate}
	\item[$(1)$] there exist an amenable, second countable, locally compact group $H$ and a continuous homomorphism $f \colon H \to G$ such that $f(H)$ is dense in~$G$, or
	\item[$(2)$] $G$ admits a chain of compact subgroups whose union is dense in~$G$.
\end{enumerate}  \end{definition}

It is not easy to provide examples of amenable Polish groups which are not F\o lner amenable in the sense of Rosendal -- however with the help of Ian Agol we managed to provide an example as follows. Let $m$ be fixed and let $G(m)$ be the inverse limit of the free $m$-generator $n$-step nilpotent groups $G(m,n)$. It is well-known that $G(m,n)$ is torsionfree for all $m, n \in \mathbb N$. Then $G(m)$ is an amenable pro-discrete Polish group. Indeed, the inverse limit of any directed family of amenable topological groups, whose limit projections have dense images, is again amenable~\cite{Pestov12}, and in particular the inverse limit of any directed family of amenable discrete groups is amenable as a topological group. 

For a fixed prime $p$, the groups $G(m)$ surjects onto $K(m,p)$, the free $m$-generator pro-$p$ group, since all finite $p$-groups are nilpotent. Now, the congruence kernel $$L(p):= \ker({\rm SL}(2,{\mathbb Z}_p) \to {\rm SL}(2,{\mathbb Z}/p{\mathbb Z}))$$ is a finitely generated pro-$p$-group and hence $K(m)$ surjects onto it for some $m$. Hence, any dense subgroup of $G(m)$ for $m$ large enough will give rise to a dense subgroup of $L(p)$. Since $L(p)$ is non-solvable, this image is non-solvable as well and thus contains a free subgroup by the Tits alternative. We get that any dense subgroup of $G(m)$ must have  a free subgroup and hence be non-amenable. An additional argument shows that $m=2$ is enough for the general strategy to work. On the other side, we claim that if $G(m)$ were F\o lner amenable, then it must have a dense subgroup which is amenable as a discrete group. Indeed, since $G(m,n)$ is torsionfree, the group $G(m)$ does not have any non-trivial compact subgroups. Thus, condition (2) cannot be satisfied and in condition (1) any continuous homomorphism from an amenable locally compact group must factor through a discrete quotient. This finishes the proof that $G(m)$ is not F\o lner amenable.

 If $G$ is an amenable Polish group not satisfying condition~(1) of Definition~\ref{definition:rosendal}, then $G \times \mathbb{Z}$ is an amenable Polish group not being F\o lner amenable: by our hypothesis on $G$, it follows that $G \times \mathbb{Z}$ does not satisfy condition~(1) of Definition~\ref{definition:rosendal}, and since $\mathbb{Z}$ is not locally finite, $G \times \mathbb{Z}$ cannot satisfy condition~(2) of Definition~\ref{definition:rosendal} either. Therefore, in order to produce more examples of amenable Polish groups which are not F\o lner amenable, it would suffice to find an amenable Polish group not satisfying condition~(1) of Definition~\ref{definition:rosendal}. For $\Aut (\mathbb{Q},{<})$, equipped with the topology of point-wise convergence, $U(\ell^{2}\mathbb{N})$  or the Fredholm unitary group $U_{C}(\ell^{2}\mathbb{N})$, endowed with the uniform operator topology, this problem can be reduced to the question of existence of a countable dense subgroup which is amenable as a discrete group. For $\Aut (\mathbb{Q},{<})$ this seems to be an interesting open problem -- likely equivalent to the famous question if Thompson's group $F$ is amenable. However, this question is also interesting and just as natural for other polish groups, such as the unitary group $U(\ell^{2}\mathbb{N})$ with the strong operator topology, or the Fredholm unitary group $U_{C}(\ell^{2}\mathbb{N})$. Let us note in this respect that there is no \emph{local} obstruction to amenabiliy of a dense subgroup in $U(\ell^{2}\mathbb{N})$; what we mean by this is that it is known that for any $k$, the set of $k$-tuples of unitaries in $U(\ell^{2}\mathbb{N})$ which generate an amenable group (in fact a finite group) is dense in the product topology. This follows from work of Lubotzky-Shalom~\cite{LubotzkyShalom}. On the other side, note that the group $U(n)$ cannot have a dense subgroup which is amenable as a discrete group -- this again is a basic consequence of the Tits alternative.
 
 \section*{Acknowledgments}

The research leading to these results has received funding from the European Research Council under the European Union's Seventh Framework Programme (FP7/2007-2013), ERC grant agreement n$^\circ$ 277728.
The first author acknowledges funding of the German Research Foundation (reference no.~SCHN 1431/3-1) as well as the Excellence Initiative by the German Federal and State Governments. We thank Ian Agol for allowing us to include the example of an amenable topological group which is not F\o lner amenable in the sense of Rosendal.



\begin{thebibliography}{CSGH99}
	
\bibitem[BB11]{babenko}
Ivan~K. Babenko and Semeon~A. Bogaty{\u\i}, \emph{Amenability of the substitution group of formal power series}, Izv. Ross. Akad. Nauk Ser. Mat. \textbf{75} (2011), no.~2, pp.~19--34.
	
\bibitem[BMP14]{BarrosoMbomboPestov}
Cleon~S. Barroso, Brice~R. Mbombo, and Vladimir~G. Pestov, \emph{On topological groups with an approximate fixed point property}, October 2014, arXiv: \url{1410.8370 [math.GR]}.

\bibitem[CT16]{CarderiThom}
Alessandro Carderi and Andreas~B. Thom, \emph{An exotic group as limit of finite special linear groups}, March 2016, arXiv: \url{1603.00691 [math.GR]}.

\bibitem[CSGH99]{ParadoxicalDecompositions}
Tullio Ceccherini-Silberstein, Rostislav~I. Grigorchuk, and Pierre de~la~Harpe, \emph{Amenability and paradoxical decompositions for pseudogroups and discrete metric spaces}, Proc. Steklov Inst. Math. \textbf{224} (1999), pp.~57--97.

\bibitem[Day57]{day57}
Mahlon~M. Day, \emph{Amenable semigroups}, Illinois J. Math. \textbf{1} (1957), pp.~509--544.

\bibitem[F{\o}l55]{folner}
Erling F{\o}lner, \emph{On groups with full {B}anach mean value}, Math. Scand. \textbf{3} (1955), pp.~243--254.

\bibitem[GL09]{GaboriauLyons}
Damien Gaboriau and Russell Lyons, \emph{A measurable-group-theoretic solution to von {N}eumann's problem}, Invent. Math. \textbf{177} (2009), no.~3, pp.~533--540.

\bibitem[GM83]{GromovMilman}
Michail Gromov and Vitali~D. Milman, \emph{A topological application of the isoperimetric inequality}, Amer. J. Math. \textbf{105} (1983), no.~4, pp.~843--854.

\bibitem[Hal35]{Hall35}
Philip Hall, \emph{{On representatives of subsets}}, Journal of the London Mathematical Society \textbf{10} (1935), pp.~26--30.

\bibitem[LS04]{LubotzkyShalom}
Alexander Lubotzky and Yehuda Shalom, \emph{Finite representations in the unitary dual and {R}amanujan groups}, In: \emph{Discrete geometric analysis}, Contemp. Math. \textbf{347}, Amer. Math. Soc., Providence, RI, 2004, pp.~173--189.

\bibitem[MU16]{MarksUnger}
Andrew Marks and Spencer Unger, \emph{Baire measurable paradoxical decompositions via matchings}, Adv. Math. \textbf{289} (2016), pp.~397--410.

\bibitem[Nam64]{namioka}
Isaac Namioka, \emph{F{\o}lner's conditions for amenable semi-groups}, Math. Scand. \textbf{15} (1964), pp.~18--28.

\bibitem[NP11]{NaorPeres}
Assaf Naor and Yuval Peres, \emph{{$L_p$} compression, traveling salesmen, and stable walks}, Duke Math. J. \textbf{157} (2011), no.~1, pp.~53--108.

\bibitem[NPS15]{NeufangPachlPekka}
Matthias Neufang, Jan Pachl, and Pekka Salmi, \emph{Uniform equicontinuity, multiplier topology and continuity of convolution}, Arch. Math. (Basel) \textbf{104} (2015), no.~4, pp.~367--376.

\bibitem[Neu29]{vonNeumann}
John von Neumann, \emph{\"{U}ber die analytischen {E}igenschaften von {G}ruppen linearer {T}ransformationen und ihrer {D}arstellungen}, Math. Z. \textbf{30} (1929), no.~1, pp.~3--42.

\bibitem[Ol'80]{olshanskii}
Alexander~Ju. Ol'{\v{s}}anski{\u\i}, \emph{On the question of the existence of an invariant mean on a group}, Uspekhi Mat. Nauk \textbf{35} (1980), no.~4(214), pp.~199--200.

\bibitem[Ore55]{Ore}
Oystein Ore, \emph{Graphs and matching theorems}, Duke Math. J. \textbf{22} (1955), pp.~625--639.

\bibitem[Pac13]{PachlBook}
Jan Pachl, \emph{Uniform spaces and measures}, Fields Institute Monographs \textbf{30}, Springer, New York; Fields Institute for Research in Mathematical Sciences, Toronto, 2013.

\bibitem[PS15]{PachlSteprans}
Jan Pachl and Juris Stepr{\= a}ns, \emph{Continuity of convolution on SIN groups}, July 2015, arXiv: \url{1507.07506 [math.FA]}.

\bibitem[Pes98]{pestov}
Vladimir~G. Pestov, \emph{On free actions, minimal flows, and a problem by {E}llis}, Trans. Amer. Math. Soc. \textbf{350} (1998), no.~10, pp.~4149--4165.

\bibitem[Pes06]{pestovbook}
Vladimir~G. Pestov, \emph{Dynamics of {I}nfinite-{D}imensional {G}roups: {T}he {R}amsey-{D}voretzky-{M}ilman {P}henomenon}, University Lecture Series \textbf{40},
American Mathematical Society, Providence, RI, 2006.

\bibitem[Pes08]{Pestov08}
Vladimir~G. Pestov, \emph{A theorem of {H}rushovski-{S}olecki-{V}ershik applied to uniform and coarse embeddings of the {U}rysohn metric space}, Topology Appl. \textbf{155} (2008), no.~14, pp.~1561--1575.

\bibitem[Pes12]{Pestov12}
Vladimir~G. Pestov, \emph{Review of~\cite{babenko}}, MathSciNet, MR 2830241 (2012e:43002), URL: \url{http://www.ams.org/mathscinet-getitem?mr=2830241}.

\bibitem[Ric67]{rickert}
Neil~W. Rickert, \emph{Amenable groups and groups with the fixed point property}, Trans. Amer. Math. Soc. \textbf{127} (1967), pp.~221--232.

\bibitem[Ros73]{rosenblatt}
Joseph~M. Rosenblatt, \emph{A generalization of {F}\o lner's condition}, Math. Scand. \textbf{33} (1973), no.~3, pp.~153--170.

\bibitem[Ros15a]{rosendal1}
Christian Rosendal, \emph{Coarse geometry of topological groups}, Preprint 2015, URL: \url{http://homepages.math.uic.edu/~rosendal/PapersWebsite/CoarseGeometry20}.

\bibitem[Ros15b]{rosendal2}
Christian Rosendal, \emph{Equivariant geometry of {B}anach spaces and topological groups}, Preprint 2015, URL:~\url{http://homepages.math.uic.edu/~rosendal/PapersWebsite/Equivariant-Geometry28}.

\bibitem[ST15]{SchneiderThom}
Friedrich~M. Schneider and Andreas~B. Thom, \emph{Topological matchings and amenability}, February 2015, arXiv: \url{1502.02293 [math.GR]}.

\bibitem[Ste76]{stern1}
Jacques Stern, \emph{Some applications of model theory in {B}anach space theory}, Ann. Math. Logic \textbf{9} (1976), no.~1--2, pp.~49--121.

\bibitem[Ste78]{stern2}
Jacques Stern, \emph{Ultrapowers and local properties of {B}anach spaces}, Trans. Amer. Math. Soc. \textbf{240} (1978), pp.~231--252.

\bibitem[Wei37]{weil}
Andr\'e Weil, \emph{Sur les espaces \`a structure uniforme et sur la topologie g\'en\'erale}, Act. Sci. Ind. \textbf{551} (1937), Paris.

\bibitem[Why99]{whyte}
Kevin Whyte, \emph{Amenability, bi-{L}ipschitz equivalence, and the von {N}eumann conjecture}, Duke Math. J. \textbf{99} (1999), no.~1, pp.~93--112.
	
\end{thebibliography}
\end{document}